\documentclass[11pt]{amsart}
\usepackage{amsmath}
\usepackage{amssymb}
\usepackage{amsthm}
\usepackage{latexsym}
\usepackage{graphicx}
\usepackage{hyperref}
\usepackage{enumerate}
\usepackage[all]{xy}
\usepackage{chngcntr}

\newtheorem{thm}{Theorem}[section]

\newtheorem{lem}[thm]{Lemma}
\newtheorem{prop}[thm]{Proposition}

\newtheorem{thmintro}{Theorem}

\theoremstyle{definition}

\newtheorem{rem}[thm]{Remark}
\newtheorem{cond}[thm]{Condition}

\newcommand{\N}{\mathbb N}
\newcommand{\Z}{\mathbb Z}

\newcommand{\R}{\mathbb R}
\newcommand{\C}{\mathbb C}

\newcommand{\mf}{\mathfrak}
\newcommand{\mc}{\mathcal}
\newcommand{\mb}{\mathbf}
\newcommand{\mh}{\mathbb}
\def\cC{{\mathcal C}}

\def\cL{{\mathcal L}}

\def\cE{{\mathcal E}}
\def\Irr{{\rm Irr}}
\newcommand{\mr}{\mathrm}
\newcommand{\ind}{\mathrm{ind}}
\newcommand{\enuma}[1]{\begin{enumerate}[\textup{(}a\textup{)}] {#1} \end{enumerate}}
\newcommand{\Fr}{\mathrm{Frob}}

\newcommand{\nr}{\mathrm{nr}}

\newcommand{\Rep}{\mathrm{Rep}}

\newcommand{\der}{\mathrm{der}}

\newcommand{\matje}[4]{\left(\begin{smallmatrix} #1 & #2 \\ 
#3 & #4 \end{smallmatrix}\right)}

\newcommand{\Mod}{\mathrm{Mod}}
\newcommand{\Hom}{\mathrm{Hom}}
\newcommand{\End}{\mathrm{End}}

\newcommand{\isom}{\xrightarrow{\sim}}

\newcommand{\sgn}{\mathrm{sgn}}
\newcommand{\pt}{\mathrm{pt}}
\newcommand{\IM}{\mathrm{IM}}
\newcommand{\Zy}{{Z^\circ_{G \times \C^\times} (y)}}
\newcommand{\Modf}[1]{\mathrm{Mod}_{\mathrm{fl}, #1}}
\newcommand{\Ad}{\mathrm{Ad}}
\newcommand{\pr}{\mathrm{pr}}

\newcommand{\IC}{\mathrm{IC}}

\begin{document}

\title[Graded Hecke algebras and the $p$-adic Kazhdan--Lusztig conjecture]{Graded Hecke 
algebras, constructible sheaves and the $p$-adic Kazhdan--Lusztig conjecture}
\date{\today}
\maketitle
\vspace{3mm}

\begin{center}
{\Large Maarten Solleveld} \\[1mm]
IMAPP, Radboud Universiteit Nijmegen\\
Heyendaalseweg 135, 6525AJ Nijmegen, the Netherlands \\
email: m.solleveld@science.ru.nl
\end{center}
\vspace{5mm}

\begin{abstract}
Graded Hecke algebras can be constructed in terms of equivariant cohomology and constructible
sheaves on nilpotent cones. In earlier work, their standard modules and their irreducible
modules where realized with such geometric methods. 

We pursue this setup to study properties of module categories of (twisted) graded Hecke algebras,
in particular what happens geometrically upon formal completion with respect to a central character. 
We prove a version of the Kazhdan--Lusztig conjecture for (twisted) graded Hecke algebras. It 
expresses the multiplicity of an irreducible module in a standard module as the multiplicity of
an equivariant local system in an equivariant perverse sheaf.

This is applied to smooth representations of reductive $p$-adic groups. Under some conditions, we
verify the $p$-adic Kazhdan--Lusztig conjecture from \cite{Vog}. Here the equivariant constructible
sheaves live on certain varieties of Langlands parameters. The involved conditions are checked for
substantial classes of groups and representations.
\end{abstract}

\tableofcontents

\section*{Introduction}

The importance of graded Hecke algebras stems from the multitude of ways in which they can arise:
\begin{itemize}
\item in terms of generators and relations,
\item as degenerations of affine Hecke algebras \cite{Lus-Gr},
\item from harmonic analysis and differential operators on Lie algebras \cite{Che,Opd},
\item from progenerators for representations of reductive $p$-adic groups \cite{SolEnd},
\item from constructible sheaves and equivariant cohomology \cite{Lus-Cusp1,Lus-Cusp2,AMS2},
\item from enhanced Langlands parameters for reductive $p$-adic groups \cite{AMS3}.
\end{itemize}
In the last three settings one naturally encounters slightly more general objects $\mh H$ called 
twisted graded Hecke algebras. The irreducible and standard modules of these algebras were classified 
and constructed geometrically in \cite{Lus-Cusp1,Lus-Cusp2,AMS2}. The main goal of this paper is 
to apply that setup to compute the multiplicity of an irreducible $\mh H$-module in another 
$\mh H$-module (typically a standard module).

Via \cite{SolEnd} the analogous issues for smooth representations of a reductive $p$-adic group
$\mc G (F)$ can be translated to modules over twisted graded Hecke algebras. In good cases, that
can also be related to the geometry of varieties of enhanced Langlands parameters for $\mc G (F)$,
via \cite{AMS3}. That links our goals to versions of the Kazhdan--Lusztig conjecture for 
$p$-adic groups \cite{Vog,Zel}.

We hope that this paper may contribute to the long term project of geometrization and categorification
of the (local) Langlands correspondence, for which we refer to \cite{BCHN, FaSc, Hel, Zhu}.
Roughly speaking, it is expected that a derived category of smooth $\mc G (F)$-representations is
equivalent with some derived category of coherent sheaves on a variety or stack of Langlands parameters.

In our setting the complexes of sheaves are equivariant and constructible. We showed in \cite{SolSGHA}
that the associated derived categories are naturally equivalent with derived categories of differential
graded modules over suitable twisted graded Hecke algebras $\mh H$. In particular, one can never
detect all $\mh H$-modules or all smooth $\mc G (F)$-representations with such sheaves -- most of them
are not graded. Nevertheless our setup could provide a stepping stone to relate more appropriate
sheaves to $\mc G (F)$-representations.\\

\textbf{Main results.}

Let $G$ be a complex reductive algebraic group and let $M$ be a Levi subgroup of $G$. For maximal
generality we allow disconnected versions of $G$ and $M$. Let $q\cE$ be an $M$-equivariant cuspidal
local system on a nilpotent orbit in Lie$(M)$. These data give rise to a (twisted) graded Hecke
algebra $\mh H (G,M,q\cE)$, see \cite[\S 2.1]{SolSGHA}. In the introduction and in most of the paper 
we assume that $N_G (M)$ stabilizes $q\cE$, which by \cite{AMS2} can be done without loss of generality.

Let $\mf g_N$ be the nilpotent cone in
$\mf g = \mr{Lie}(G)$. Via some kind of parabolic induction, one constructs from $q\cE$ a semisimple
complex $K_N \in \mc D^b_{G \times \C^\times} (\mf g_N)$. Here $\mc D^b_{G \times \C^\times}$ denotes
an equivariant bounded derived category of constructible sheaves, from \cite{BeLu}. 
This provides an isomorphism of graded algebras \cite[Theorem 2.2]{SolSGHA}:
\begin{equation}\label{eq:1}
\mh H (G,M,q\cE) \cong \End^*_{\mc D^b_{G \times \C^\times} (\mf g_N)} (K_N) .
\end{equation}
The algebra $\mh H (G,M,q\cE)$ comes with a finite ``Weyl-like" group $W_{q\cE}$ acting on   
$\mf t = \mr{Lie}(Z(M))$. Its centre can be described as
\[
Z( \mh H (G,M,q\cE)) \cong \mc O (\mf t \oplus \C)^{W_q\cE} = 
\mc O (\mf t / W_{q\cE}) \otimes_\C \C [\mb r] .
\]
Via a specific injection $\Sigma_v : \mf t \oplus \C \to \mf m \oplus \C$ (the identity on $\mf t$,
but in general not on $\C$), we can regard $Z( \mh H (G,M,q\cE))$ as a quotient of 
$\mc O (\mf g \oplus \C)^G$. A semisimple element $(\sigma,r) \in \mf g \oplus \C$ determines a unique
central character of $\mh H (G,M,q\cE)$ if it lies in $\mr{Ad}(G) \Sigma_v (\mf t \oplus \C)$, and
otherwise it is irrelevant for $\mh H (G,M,q\cE)$. We fix a relevant $(\sigma,r)$ and we denote the 
corresponding formal completion of $Z(\mh H (G,M,q\cE))$ by $\hat{Z}(\mh H (G,M,q\cE))_{\sigma,r}$. 
In the process of localization, $G \times \C^\times$ will be replaced by 
$Z_G (\sigma) \times \C^\times$ and $\mf g_N$ by 
\[
\mf g_N^{\sigma,r} := \{ y \in \mf g_N : [\sigma,y] = 2 r y \} .
\]
A variation on the construction of $K_N$ yields an object 
$K_{N,\sigma,r} \in \mc D^b_{Z_G (\sigma) \times \C^\times} (\mf g_N^{\sigma,r})$.
Since $(\sigma,r)$ belongs to $\mr{Lie}(Z_G (\sigma) \times \C^\times)$, it defines a character of 
\[
H^*_{Z_G (\sigma) \times \C^\times}(\pt) \cong 
\mc O \big( \mr{Lie}(Z_G (\sigma) \times \C^\times) \big)^{Z_G (\sigma) \times \C^\times} ,
\]
and we can formally complete the latter algebra with respect to $(\sigma,r)$.

\begin{thmintro}\label{thm:A}
(see Theorem \ref{thm:2.5}) \\
There is a natural algebra isomorphism
\begin{multline*}
\hat Z (\mh H (G,M,q\cE))_{\sigma,r} \underset{Z (\mh H (G,M,q\cE))}{\otimes} \mh H (G,M,q\cE) \; \isom \\
\hat{H}^*_{Z_G (\sigma) \times \C^\times} (\pt)_{\sigma,r} \underset{H_{Z_G (\sigma) \times \C^\times}^* 
(\pt)} {\otimes} \End^*_{\mc D^b_{Z_G (\sigma) \times \C^\times}(\mf g_N^{\sigma,r})} (K_{N,\sigma,r}) .
\end{multline*}
This induces an equivalence of categories
\[
\Modf{\sigma ,r} (\mh H (G,M,q\cE)) \cong \Modf{\sigma ,r} \big( \End^*_{\mc D^b_{Z_G (\sigma) \times 
\C^\times}(\mf g_N^{\sigma,r})}(K_{N,\sigma,r}) \big) .
\]
Here $\Modf{\sigma ,r}$ denotes the category of finite length modules all whose irreducible
subquotients admit the central character $(\sigma,r)$.
\end{thmintro}

When $G$ is connected and $r \neq 0$, Theorem \ref{thm:A} is due to Lusztig \cite{Lus-Cusp2}.
Our investigations revealed a technical problem in the relevant part of \cite{Lus-Cusp2}, which was 
resolved in collaboration with Lusztig, see \cite[\# 121]{Lus-Corr} and Appendix \ref{sec:B}. 

Graded Hecke algebras coming from reductive $p$-adic groups have the variable $\mb r \in \mh H (G,M,q\cE)$ 
specialized to a positive real number. In that respect a version of Theorem \ref{thm:A} for 
$\mh H (G,M,q\cE) / (\mb r - r)$ is fitting. As $\C [\mb r] = \mc O (\mr{Lie}(\C^\times))$
comes from the $\C^\times$-actions, a natural attempt is to replace $G \times \C^\times$-equivariance by 
$G$-equivariance. That does not work well directly in \eqref{eq:1}, only after localization. 

\begin{thmintro}\label{thm:B}
(see Paragraph \ref{par:fixed}) \\
Fix $r \in \C$. Theorem \ref{thm:A} becomes valid with $\mh H (G,M,q\cE) / (\mb r - r)$ instead
of $\mh H (G,M,q\cE)$ once we forget the $\C^\times$-equivariant structure everywhere.
\end{thmintro}

Theorems \ref{thm:A} and \ref{thm:B} could be helpful for a geometric construction of
Ext-groups between objects of $\Modf{\sigma ,r} (\mh H (G,M,q\cE))$ or
$\Modf{\sigma} (\mh H (G,M,q\cE) / (\mb r - r))$.

For more concrete results, it is necessary to understand standard and irreducible 
$\mh H (G,M,q\cE)$-modules better. They were obtained in \cite{Lus-Cusp1,AMS2} via the equivariant 
cohomology of certain flag varieties with local systems. They can be parametrized by the following 
data (considered up to $G$-conjugation): 
\[
\text{semisimple } \sigma \in \mf g ,\: r \in \C ,\: y \in \mf g_N^{\sigma,r} \text{ and certain }
\rho \in \Irr \big( \pi_0 (Z_G (\sigma,y)) \big).
\] 
The standard $\mh H (G,M,q\cE)$-module $E_{y,\sigma,r,\rho}$ has a distinguished (unique if $r \neq 0$)
irreducible quotient $M_{y,\sigma,r,\rho}$. In Paragraph \ref{par:involutions} we observe that the
parametrization from \cite{AMS2} is more complicated than necessary, and we make it more natural.

Next we provide several useful alternative constructions of $E_{y,\sigma,r,\rho}$:

\begin{thmintro}\label{thm:C}
(see Proposition \ref{prop:6.2} and Lemma \ref{lem:6.3}) \\
Suppose that $\rho \in \Irr \big( \pi_0 (Z_G (\sigma,y)) \big)$ fulfills the condition to parametrize
an \\ $\mh H (G,M,q\cE)$-module. Let $i_y : \{y\} \to \mf g_N^{\sigma,r}$ be the inclusion.
\enuma{
\item There is an isomorphism of $\mh H (G,M,q\cE)$-modules
\[
\Hom_{\pi_0 (Z_G (\sigma,y))} \big( \rho, H^* (i_y^! K_{N,\sigma,r}) \big) \cong E_{y,\sigma,r,\rho}.
\]
\item Denote the dual of a local system or representation by a $\vee$. There is an isomorphism of 
$\mh H (G,M,q\cE^\vee)$-modules
\[
\Hom_{\pi_0 (Z_G (\sigma,y))} \big( \rho^\vee, H^* (i_y^* K_{N,\sigma,r})^\vee \big) \cong
E_{y,\sigma,r,\rho^\vee} .
\]
\item Let $\mr{ind}_{Z_G (\sigma,y)}^{Z_G (\sigma)} \rho$ be the $Z_G (\sigma) \times \C^\times
$-equivariant local system on $\mr{Ad}(Z_G (\sigma)) y$ determined by $\rho$, and let
$j_N : \mr{Ad}(Z_G (\sigma)) y \to \mf g_N^{\sigma,r}$ be the inclusion. Then
\begin{equation}\label{eq:2}
\Hom^*_{\mc D^b_{Z_G (\sigma) \times \C^\times}(\mf g_N^{\sigma,r})} \big( K_{N,\sigma,r},
j_N^* \mr{ind}_{Z_G (\sigma,y)}^{Z_G (\sigma)} \rho \big)
\end{equation}
is a graded $\mh H (G,M,q\cE^\vee)$-module, 
\begin{equation}\label{eq:3}
\Hom^*_{\mc D^b_{Z_G (\sigma)}(\mf g_N^{\sigma,r})} \big( K_{N,\sigma,r},
j_N^* \mr{ind}_{Z_G (\sigma,y)}^{Z_G (\sigma)} \rho \big)
\end{equation}
is a graded right $\End^*_{\mc D^b_{Z_G (\sigma)}(\mf g_N^{\sigma,r})}(K_{N,\sigma,r})$-module and
there are canonical surjections of $\mh H (G,M,q\cE^\vee)$-modules
\[
\eqref{eq:2} \to \eqref{eq:3} \to E_{y,\sigma,r,\rho^\vee} .
\]
}
\end{thmintro}

We note that $E_{y,\sigma,r,\rho^\vee}$ is not a graded $\mh H (G,M,q\cE^\vee)$-module, because the 
surjection in Theorem \ref{thm:C}.c consists of dividing out the submodule generated by the 
inhomogeneous ideal 
\[
\ker (\mr{ev}_{\sigma,r}) \subset H^*_{Z_G (\sigma) \times \C^\times}(\pt) \cong
\mc O (Z_{\mf g}(\sigma))^{Z_G (\sigma)} \otimes_\C \C [\mb r] .
\]
In Appendix \ref{sec:parabolic} we prove that, under a mild condition, standard modules of twisted
graded Hecke algebras are compatible with parabolic induction. That is relevant for parts of 
\cite{AMS2,AMS3} which are related to Section \ref{sec:KL}, and it shows that our standard modules
are really analogous to the standard representations for reductive groups studied by Langlands
and others.

The original Kazhdan--Lusztig conjecture \cite{KaLu1} concerned the multiplicities of irreducible
modules in standard modules for semisimple complex Lie algebras. One can ask for the analogous 
multiplicities for any group or algebra with a good notion of standard modules, in particular for
a (twisted) graded Hecke algebra. The next result relies on some properties of $(\mf g_N^{\sigma,r},
K_{N,\sigma,r})$ from Section \ref{sec:structure}.

\begin{thmintro}\label{thm:D} 
(see Proposition \ref{prop:2.9}) \\
Suppose that $(y,\sigma,r,\rho)$ and $(y',\sigma,r,\rho')$ parametrize $\mh H (G,M,q\cE)$-modules.
The multiplicity of the irreducible module $M_{y',\sigma,r,\rho'}$ in the standard module
$E_{y,\sigma,r,\rho}$ equals the multiplicity of the local system 
$\mr{ind}_{Z_G (\sigma,y)}^{Z_G (\sigma)} \rho$ in the pullback to $\mr{Ad}(Z_G (\sigma))y$ of
the cohomology sheaf $\mc H^* \big( \IC_{Z_G (\sigma)} (\mf g_N^{\sigma,r}, 
\mr{ind}_{Z_G (\sigma,y')}^{Z_G (\sigma)} \rho') \big)$.
\end{thmintro}

A version of the Kazhdan--Lusztig conjecture for the $p$-adic group $GL_n (F)$ appeared in \cite{Zel},
and Vogan \cite{Vog} formulated it for all connected reductive groups over local fields. We point out
that \cite[Conjecture 8.11]{Vog} contains some signs which are useful for real reductive groups but
better omitted in the non-archimedean instances.

To transfer Theorem \ref{thm:D} to reductive $p$-adic groups and their Langlands parameters, we need
to make several assumptions about aspects of the local Langlands correspondence. We refer to Section
\ref{sec:KL} for an explanation of the setup and the conditions.

Let $q_F$ be the cardinality of the residue field of the non-archimedean local field $F$. 
For $r = \log (q_F)/ 2$ the variety
\begin{equation}\label{eq:4}
\mf g_N^{\sigma,-r} = \{ y \in \mf g_N : \Ad (\exp \sigma) y = q_F^{-1} y \} .
\end{equation}
can be identified with a set of unramified Langlands parameters $\phi : \mb W_F \rtimes \C \to G$, 
with $\exp (\sigma)$ the image of a Frobenius element of $\mb W_F$. This can be used to model
varieties of arbitrary Langlands parameters for $\mc G (F)$.

\begin{thmintro}\label{thm:E} 
(see Theorem \ref{thm:8.7}) \\
Consider a Bernstein block in the category of smooth complex representations of a connected reductive 
$p$-adic group $\mc G (F)$, coming from a supercuspidal representation $\omega$ of a Levi subgroup 
$\mc M (F)$. Suppose that the conditions from Section \ref{sec:KL} hold.
\enuma{
\item Let $\Rep_{\mr{fl}} (\mc G (F))^\omega$ be the category of finite length $\mc G (F)
$-representations all whose irreducible subquotients have cuspidal support $(\mc M (F),\omega)$. 

Then $\Rep_{\mr{fl}} (\mc G (F))^\omega$ is equivalent with a category of the form
\[
\Modf{\sigma} \big( \End^*_{\mc D^b_{Z_G (\sigma)}(\mf g_N^{\sigma,-r})}(K_{N,\sigma,-r}) \big)
\quad \text{where} \quad r = \log (q_F) /2 . \vspace{-1mm}
\]
\item The $p$-adic Kazhdan--Lusztig conjecture (as in \cite[Conjecture 8.11]{Vog} without the signs)
holds for irreducible and standard representations in $\Rep_{\mr{fl}} (\mc G (F))^\omega$. It takes the 
form of Theorem \ref{thm:D}, where all objects live on a variety of Langlands parameters \eqref{eq:4}.
\item Parts (a) and (b) hold unconditionally in the following cases:
\begin{itemize}
\item inner forms of general linear groups,
\item inner forms of special linear groups,
\item principal series representations of quasi-split groups,
\item unipotent representations (of arbitrary connected reductive groups over $F$),
\item classical $F$-groups (not necessarily quasi-split).
\end{itemize}
}
\end{thmintro}

\textbf{Acknowledgements}\\
We thank George Lusztig, Eugen Hellmann and David Ben-Zvi for some helpful conversations.
We are very grateful to the referee for his of her detailed reports, which lead to many
improvements in the paper.

\renewcommand{\theequation}{\arabic{section}.\arabic{equation}}
\counterwithin*{equation}{section}

\section{The setup from \cite{SolSGHA}}
\label{sec:setup}

All our groups will be complex linear algebraic groups. We mainly work in the equivariant bounded
derived categories of constructible sheaves from \cite{BeLu}. For a group $H$ acting on a space $X$,
this category will be denoted $\mc D^b_H (X)$.

Let $G$ be a complex reductive group, possibly disconnected. To construct a graded Hecke algebra
geometrically, we need a cuspidal quasi-support $(M,\cC_v^M,q\cE)$ for $G$ \cite{AMS1}. This consists of:
\begin{itemize}
\item a quasi-Levi subgroup $M$ of $G$, which means that $M^\circ$ is a Levi subgroup of $G^\circ$ and
$M = Z_G (Z (M^\circ)^\circ)$,
\item $\cC_v^M$ is a Ad$(M)$-orbit in the nilpotent variety $\mf m_N$ in the Lie algebra $\mf m$ of $M$,
\item $q\cE$ is a $M$-equivariant cuspidal local system on $\cC_v^M$.
\end{itemize}
We write $T = Z(M)^\circ$, $\mf t = \mr{Lie}(T)$ and 
\[
W_{q\cE} = \mr{Stab}_{N_G (M)} (q\cE) / M = N_G (M,q\cE) / M
\]
To these data one associates a twisted graded Hecke algebra
\begin{equation}\label{eq:1.3}
\mh H (G,M,q\cE) = \mh H (\mf t,W_{q\cE},k,\mb r,\natural_{q\cE}] ,
\end{equation}
see \cite[\S 2.1]{SolSGHA}. As vector space it is the tensor product of 
\begin{itemize}
\item a polynomial algebra $\mc O (\mf t \oplus \C) = \mc O (\mf t) \otimes \C [\mb r]$,
\item a twisted group algebra $\C [W_{q\cE},\natural_{q\cE}]$,
\end{itemize}
and there are nontrivial cross relations between these two subalgebras.

Let $\mf g_N$ be the nilpotent variety in the Lie algebra $\mf g$ of $G$. The algebra \eqref{eq:1.3} can
be realized in terms of suitable equivariant sheaves on $\mf g$ or $\mf g_N$. We let $\C^\times$ act
on $\mf g$ and $\mf g_N$ by $\lambda \cdot X = \lambda^{-2} X$. Then every $M$-equivariant local system 
on $\cC_v^M$, and in particular $q\cE$, is automatically $M \times \C^\times$-equivariant.

Let $P^\circ = M^\circ U$ be a parabolic subgroup of $G^\circ$ with Levi factor $M^\circ$ and unipotent
radical $U$. Then $P = M U$ is a ``quasi-parabolic" subgroup of $G$. Consider the varieties
\begin{align*}
& \dot{\mf g} = \{ (X,gP) \in \mf g \times G/P : 
\Ad (g^{-1}) X \in \cC_v^M \oplus \mf t \oplus \mf u \} , \\
& \dot{\mf g}_N = \dot{\mf g} \cap (\mf g_N \times G/P) .
\end{align*}
We let $G \times \C^\times$ act on these varieties by
\[
(g_1,\lambda) \cdot (X,gP) = (\lambda^{-2} \Ad (g_1) X, g_1 g P) .
\] 
By \cite[Proposition 4.2]{Lus-Cusp1} there are natural isomorphisms of graded algebras
\begin{equation}\label{eq:1.1}
H^*_{G \times \C^\times} (\dot{\mf g}) \cong H^*_{G \times \C^\times} (\dot{\mf g}_N) 
\cong \mc O (\mf t) \otimes_\C \C [\mb r] .
\end{equation}
Consider the maps
\begin{equation}\label{eq:1.2}
\begin{aligned}
& \cC_v^M \xleftarrow{f_1} \{ (X,g) \in \mf g \times G : \Ad (g^{-1}) X \in 
\cC_v^M \oplus \mf t \oplus \mf u\} \xrightarrow{f_2} \dot{\mf g} , \\
& f_1 (X,g) = \mathrm{pr}_{\cC_v^M}(\Ad (g^{-1}) X) , \hspace{2cm} f_2 (X,g) = (X,gP) .
\end{aligned}
\end{equation}
Let $\dot{q\cE}$ be the unique $G \times \C^\times$-equivariant local system on $\dot{\mf g}$ 
such that $f_2^* \dot{q\cE} = f_1^* q\cE$. Let $\mr{pr}_1 : \dot{\mf g} \to \mf g$ be the projection 
on the first coordinate and define
\[
K := \pr_{1,!} \dot{q\cE} \qquad \in \mc D^b_{G \times \C^\times} (\mf g ) .
\]
Let $\dot{q\cE}_N$ be the pullback of $\dot{q\cE}$ to $\dot{\mf g}_N$ and put
\[
K_N := \pr_{1,N,!} \dot{q\cE}_N \qquad \in \mc D^b_{G \times \C^\times} (\mf g_N ) ,
\]
a semisimple complex isomorphic to the pullback of $K$ to $\mf g_N$ \cite[\S 2.2]{SolSGHA}.
From \cite[Theorem 2.2]{SolSGHA}, based on \cite{Lus-Cusp1,Lus-Cusp2,AMS2}, we recall:

\begin{thm}\label{thm:1.2}
There exist natural isomorphisms of graded algebras
\[
\mh H (G,M,q\cE) \longrightarrow \End^*_{\mc D^b_{G \times \C^\times}(\mf g)}(K) \longrightarrow
\End^*_{\mc D^b_{G \times \C^\times}(\mf g_N)}(K_N) .
\]
\end{thm}

Consider the subgroup $N_G (M,q\cE) G^\circ = N_G (P,M,q\cE) G^\circ$ of $G$. It is known from
\cite[(90)]{AMS2} that 
\[
\mh H (G,M,q\cE) = \mh H (N_G (P,M,q\cE) G^\circ, M ,q\cE) .
\]
Moreover, by \cite[Lemma 3.21]{AMS2} the relevant sets of parameters for these two algebras are
in natural bijection. Therefore we may, and will, assume without loss of generality:
\begin{cond}\label{cond:2.1}
$G$ equals $N_G (P,M,q\cE) G^\circ$, or equivalently $N_G (M)$ stabilizes $q\cE$.
\end{cond}

\section{Formal completion at a central character}
\label{sec:completion}

We want to complete $\mh H (G,M,q\cE) \cong \End^*_{\mc D^b_{G \times \C^\times}(\mf g)}(K)$ with respect
to (the kernel of) a central character. Recall from \cite[Lemma 2.3 and \S 4]{AMS2} that 
\begin{equation}\label{eq:2.1}
Z(\mh H (G,M,q\cE)) \supset \mc O (\mf t \oplus \C)^{W_{q\cE}} = 
\mc O (\mf t / W_{q \cE}) \otimes_\C \C [\mb r] .
\end{equation}
In many cases \eqref{eq:2.1} is an equality, namely whenever $W_{q\cE}$ acts faithfully on $\mf t$. 

\begin{lem}\label{lem:2.6}
The algebra $\End^*_{\mc D^b_{G^\circ \times \C^\times}(\mf g)}(K) \cong 
\End^*_{\mc D^b_{G^\circ \times \C^\times}(\mf g_N)}(K_N)$ is finitely generated as a module over
the Noetherian ring $\mc O (\mf t \oplus \C)^{W_{q\cE}}$.
\end{lem}
\begin{proof}
With minor variations, the arguments for Theorem \ref{thm:1.2}, based on \cite{Lus-Cusp1,Lus-Cusp2,AMS2, 
SolSGHA}, also apply here. They show that the two algebras in the statement are isomorphic, and that 
\begin{equation}\label{eq:2.14}
\End^*_{\mc D^b_{G^\circ \times \C^\times}(\mf g)}(K) \cong \mc O (\mf t \oplus \C) \otimes 
\End^0_{\mc D^b_{G^\circ \times \C^\times}(\mf g)}(K)
\end{equation}
as $\mc O (\mf t \oplus \C)$-modules. Since $K$ is a finite direct sum of analogous objects for $G^\circ$,
the algebra $\End^0_{\mc D^b_{G^\circ \times \C^\times}(\mf g)}(K)$ has finite dimension. Hence
\eqref{eq:2.14} is finitely generated as module over $\mc O (\mf t \oplus \C)$. The latter is a finitely 
generated Noetherian algebra and it is integral over $\mc O (\mf t \oplus \C)^{W_{q\cE}}$, so has 
finite rank over $\mc O (\mf t \oplus \C)^{W_{q\cE}}$. 
\end{proof}

To localize in a geometric way, we need to interpret \eqref{eq:2.1} in terms of equivariant homology.
By \cite[\S 1.11]{Lus-Cusp1} there are natural isomorphisms
\begin{equation}\label{eq:2.2}
H_{G \times \C^\times}^* (\mr{pt}) \cong \mc O (\mf g \oplus \C)^{G \times \C^\times} \cong
\mc O (\mf g /\!/ G) \otimes_\C \C [\mb r] .
\end{equation}
The algebra $H_{G \times \C^\times}^* (\mr{pt})$ acts naturally on
$\End^*_{\mc D^b_{G \times \C^\times}(\mf g)}(K)$
by the product in equivariant homology \cite[\S 1.20]{Lus-Cusp2}. That determines a homomorphism
\begin{equation}\label{eq:2.25}
H_{G \times \C^\times}^* (\mr{pt}) \to Z \big( \End^*_{\mc D^b_{G \times \C^\times}(\mf g)}(K) \big)
\cong  \mc O (\mf t \oplus \C)^{W_{q\cE}} .
\end{equation}
Let $\gamma_v : SL_2 (\C) \to M$ be an algebraic homomorphism with d$\gamma_v \matje{0}{1}{0}{0} = v$.
With $\sigma_v := \textup{d} \gamma_v \matje{1}{0}{0}{-1} \in \mf m$ we define an injection
\[
\begin{array}{cccc}
\Sigma_v : & \mf t \oplus \C & \to & \mf m \oplus \C  \\
& (\sigma_0,r) & \mapsto & (\sigma_0 + r \sigma_v,r) 
\end{array} .
\] 

\begin{lem}\label{lem:2.2}
\enuma{
\item The map 
\[
\mf t / W_{q \cE} \times \C = \Irr (\mc O (\mf t \oplus \C)^{W_{q\cE}}) \to 
\Irr ( H^*_{G \times \C^\times}(\mr{pt})) = \mf g /\!/ G \times \C 
\]
dual to \eqref{eq:2.25} is injective and equals the map induced by $\Sigma_v$.
\item \hspace{-2mm} The support of $\End^*_{\mc D^b_{G \times \C^\times}(\mf g)}(K)$ as 
$H_{G \times \C^\times}^* (\mr{pt})$-module is $\mr{Ad}(G) \Sigma_v (\mf t \oplus \C) /\!/ \mr{Ad}(G)$.
}
\end{lem}
\begin{proof}
(a) This follows from \cite[Proposition 1.4]{AMS3} upon specializing all coordinates from
$\vec{r}$ to $r$.\\
(b) Let $T_M$ be a maximal torus of $M^\circ$, whose Lie algebra $\mf t_M$ contains 
$\mf t \oplus \C \sigma_v$. By part (a) 
\[
(\mf g \oplus \C) /\!/ G \cong (\mf t_M \oplus \C) / W(G^\circ ,T_M) .
\] 
Since $\Sigma_v (\mf t \oplus \C)$ is a closed subset of $\mf t_M \oplus \C$, 
$\mr{Ad}(G) \Sigma_v (\mf t \oplus \C) /\!/ \mr{Ad}(G)$ is closed in $(\mf g \oplus \C) /\!/ G$.
Now the statement is a consequence of \eqref{eq:2.1}, \eqref{eq:2.2} and part (a). 
\end{proof}

Lemma \ref{lem:2.2} entails that we can formally complete 
$\End^*_{\mc D^b_{G \times \C^\times}(\mf g)}(K)$ with
respect to elements of $\mf g /\!/ G \times \C$ that come from $\mf t / W_{q\cE} \times \C$.
With the techniques from \cite[\S 4]{Lus-Cusp2}, that can be done geometrically. In Appendix
\ref{sec:B} we discuss why these techniques apply in the setting of \cite[\S 8]{Lus-Cusp2}, 
which is a special case of our current setting.

Fix $(\sigma,r) \in \Ad(G) \Sigma_v (\mf t \times \C)$ and put
\[
C = Z_{G \times \C^\times} (\sigma,r) = Z_G (\sigma) \times \C^\times .
\]
The inclusion $\mf c = \mr{Lie}(C) \subset \mf g$ makes $H_C^* (\pt)$ into a module for
\[
H_{G \times \C^\times}^* (\pt) = \mc O (\mf g)^G \otimes \C [\mb r] .
\]
Further, any $G \times \C^\times$-equivariant sheaf can be regarded as a $C$-equivariant sheaf. 
Like in \eqref{eq:2.25} we obtain a graded algebra homomorphism
\[
H_C^* (\pt) \to \End^*_{\mc D^b_C (\mf g)}(K) .
\]
That induces a graded algebra homomorphism
\begin{equation}\label{eq:2.4}
H_C^* (\pt) \underset{H_{G \times \C^\times}^* (\pt)}{\otimes} \End^*_{\mc D^b_{G \times \C^\times}
(\mf g)}(K) \longrightarrow H_C^* (\pt) \underset{H_C^* (\pt)}{\otimes} \End^*_{\mc D^b_{C}(\mf g)}(K) .
\end{equation}
We denote the completion of $H_{G \times \C^\times}^* (\pt)$ with respect to the maximal ideal 
determined by $(\sigma,r)$ as 
\[
\hat{H}^*_{G \times \C^\times} (\pt)_{\sigma,r} = 
\hat{\mc O} (\mf g \oplus \C /\!/ G \times \C^\times)_{\sigma,r}
\] 
We define $\hat{H}^*_C (\pt)_{\sigma,r} = \hat{\mc O}(\mf c /\!/ C)_{\sigma,r}$ analogously. 

\begin{prop}\label{prop:2.3}
\enuma{
\item The natural map $H^*_{G \times \C^\times}(\pt) \to H_C^* (\pt)$ induces an algebra isomorphism
$\hat{H}^*_{G \times \C^\times} (\pt)_{\sigma,r} \to \hat{H}^*_C (\pt)_{\sigma,r}$.
\item Part (a) and \eqref{eq:2.4} induce an isomorphism of $\hat{H}^*_C (\pt)_{\sigma,r}$-algebras
\[
\hat{H}^*_{G \times \C^\times} (\pt)_{\sigma,r} \underset{H_{G \times \C^\times}^* (\pt)}{\otimes} 
\End^*_{\mc D^b_{G \times \C^\times}(\mf g)}(K) \longrightarrow \hat{H}^*_C (\pt)_{\sigma,r}
\underset{H_C^* (\pt)}{\otimes} \End^*_{\mc D^b_{C}(\mf g)}(K) .
\]
\item The graded algebra $\End^*_{\mc D^b_{C}(\mf g)}(K)$ is Noetherian and only has terms in
even degrees $\geq 0$.
\item Parts (b) and (c) also hold with $(\mf g_N, K_N)$ instead of $(\mf g, K)$.
}
\end{prop}
\begin{proof}(a)
According to \cite[4.3.(a)]{Lus-Cusp2} this holds for connected groups, so for 
$G^\circ \times \C^\times$ and $C^\circ = Z_G (\sigma)^\circ \times \C^\times$. 
From $H_G^* (\pt) = H_{G^\circ}^* (\pt)^{G / G^\circ}$ \cite[\S 1.9]{Lus-Cusp1} we deduce that 
\begin{equation}\label{eq:2.5}
\begin{aligned}
\hat{H}^*_{G \times \C^\times} (\pt)_{\sigma,r} & 
= \Big( \bigoplus\nolimits_{g \in G / Z_G (\sigma) G^\circ} 
\hat{H}^*_{G \times \C^\times} (\pt)_{\Ad (g)\sigma,r} \Big)^{G / G^\circ} \\
& \cong (\hat{H}^*_{G^\circ \times \C^\times} (\pt)_{\sigma, r} )^{Z_G (\sigma) G^\circ / G^\circ} 
\; \cong \; ( \hat{H}^*_{C^\circ} (\pt)_{\sigma,r} )^{Z_G (\sigma)} \\
& = ( \hat{H}^*_{C^\circ} (\pt)_{\sigma,r} )^{C / C^\circ} \; = \; \hat{H}^*_{C} (\pt)_{\sigma,r} .
\end{aligned}
\end{equation}
(b) Part (a) and Proposition \ref{prop:B.1} show this for connected algebraic groups: there
is a natural $C$-equivariant isomorphism of $\hat H^*_{C^\circ} (\pt)_{\sigma,r}$-algebras
\begin{equation}\label{eq:2.17}
\hat{H}^*_{G^\circ \times \C^\times} (\pt)_{\sigma,r} \underset{H_{G^\circ \times \C^\times}^* 
(\pt)}{\otimes} \End^*_{\mc D^b_{G^\circ \times \C^\times}(\mf g)}(K) \to \hat{H}^*_{C^\circ} 
(\pt)_{\sigma,r} \underset{H_{C^\circ}^* (\pt)}{\otimes} \End^*_{\mc D^b_{C^\circ}(\mf g)}(K) .
\end{equation}
By Lemma \ref{lem:2.6}, $\End^*_{\mc D^b_{G^\circ \times \C^\times}(\mf g)}(K)$ has finite rank 
over $\mc O (\mf t \oplus \C)^{W_{q\cE}}$. Lemma \ref{lem:2.2}.a implies that \eqref{eq:2.25} is 
surjective, so $\End^*_{\mc D^b_{G^\circ \times \C^\times}(\mf g)}(K)$ also has 
finite rank as $H^*_{G \times \C^\times} (\pt)$-module. Consequently both sides of \eqref{eq:2.17}
are finitely generated over $\hat{H}^*_{G^\circ \times \C^\times} (\pt)_{\sigma,r}$.
Consider the $C$-stable maximal ideal 
\[
J = \ker (\mr{ev}_{\sigma,r}) \subset \hat{H}^*_{G^\circ \times \C^\times} (\pt)_{\sigma,r} 
\cong \hat{H}^*_{C^\circ} (\pt)_{\sigma,r}. 
\]
For $L \in \{ C^\circ,C, G^\circ \times \C^\times, Z_G (\sigma) G^\circ \times \C^\times \}$ we write
\[
J_L = J \cap H^*_L (\pt) = \ker (\mr{ev}_{\sigma,r} : H^*_L (\pt) \to \C) .
\]
For any $n \in \N$ we can divide out the submodules generated by $J^n$ on both sides 
of \eqref{eq:2.17}, which yields $C$-equivariant isomorphisms 
\begin{equation}\label{eq:2.6}
\begin{aligned}
& \End^*_{\mc D^b_{G^\circ \times \C^\times}(\mf g)}(K) \big/ J^n_{G^\circ \times \C^\times} 
\End^*_{\mc D^b_{G^\circ \times \C^\times}(\mf g)}(K) \cong \\
& \hat{H}^*_{G^\circ \times \C^\times} (\pt)_{\sigma,r} \! \underset{H_{G^\circ \times \C^\times}^* 
(\pt)}{\otimes} \hspace{-2mm} \End^*_{\mc D^b_{G^\circ \times \C^\times}(\mf g)}(K) \big/
J^n_{G^\circ \times \C^\times} \! \underset{H_{G^\circ \times \C^\times}^* 
(\pt)}{\otimes} \hspace{-2mm} \End^*_{\mc D^b_{G^\circ \times \C^\times}(\mf g)}(K) \hspace{-3mm} \\
& \longrightarrow \hat{H}^*_{C^\circ} (\pt)_{\sigma,r} \underset{H_{C^\circ}^* (\pt)}{\otimes} 
\End^*_{\mc D^b_{C^\circ}(\mf g)}(K) \big/ J^n_{C^\circ} \underset{H_{C^\circ}^* (\pt)}{\otimes} 
\End^*_{\mc D^b_{C^\circ}(\mf g)}(K) \cong \\
& \hspace{9mm} \End^*_{\mc D^b_{C^\circ}(\mf g)}(K) \big/ J^n_{C^\circ} \End^*_{\mc D^b_{C^\circ}(\mf g)}(K) .
\end{aligned}
\end{equation}
The finite group $C / C^\circ \cong Z_G (\sigma) G^\circ / G^\circ$ acts naturally on all terms in
\eqref{eq:2.6}. Ave\-ra\-ging over this group, we obtain an isomorphism 
\begin{equation}\label{eq:2.7}
\begin{aligned}
& \End^*_{\mc D^b_{Z_G (\sigma) G^\circ \times \C^\times}(\mf g)}(K) / J^n_{Z_G (\sigma) G^\circ 
\times \C^\times} \End^*_{\mc D^b_{Z_G (\sigma) G^\circ \times \C^\times}(\mf g)}(K) \cong \\
& \Big( \End^*_{\mc D^b_{G^\circ \times \C^\times}(\mf g)}(K) / J^n_{G^\circ \times \C^\times}
\End^*_{\mc D^b_{G^\circ \times \C^\times}(\mf g)}(K) \Big)^{Z_G (\sigma) G^\circ / G^\circ} \longrightarrow \\
& \Big( \End^*_{\mc D^b_{C^\circ}(\mf g)}(K) / J^n_{C^\circ} \End^*_{\mc D^b_{C^\circ}(\mf g)}(K) 
\Big)^{C / C^\circ} \cong \\
& \End^*_{\mc D^b_{C}(\mf g)}(K) / J^n_{C} \End^*_{\mc D^b_{C}(\mf g)}(K) .
\end{aligned}
\end{equation}
Since we are dealing with finitely generated modules, the inverse limit of the instances of 
\eqref{eq:2.7} for $n \in \N$ is a natural isomorphism of $\hat H^*_{C} (\pt)_{\sigma,r}$-algebras
\begin{multline}\label{eq:2.3}
\hat{H}^*_{Z_G (\sigma) G^\circ \times \C^\times} (\pt)_{\sigma,r} \underset{H_{Z_G (\sigma) G^\circ 
\times \C^\times}^* (\pt)}{\otimes} \End^*_{\mc D^b_{Z_G (\sigma) G^\circ \times \C^\times}(\mf g)}(K) \\
\longrightarrow \quad
\hat{H}^*_C (\pt)_{\sigma,r} \underset{H_C^* (\pt)}{\otimes} \End^*_{\mc D^b_{C}(\mf g)}(K) .
\end{multline}
A computation analogous to \eqref{eq:2.5} shows that the domain of \eqref{eq:2.3} is 
naturally isomorphic to
\[
\hat{H}^*_{G \times \C^\times} (\pt)_{\sigma,r} \underset{H_{G \times \C^\times}^* (\pt)}{\otimes} 
\End^*_{\mc D^b_{G \times \C^\times}(\mf g)}(K) .
\]
(c) As in the proof of part (b) one sees that $\mh H (G,M,q\cE)$ has 
finite rank as $H^*_{G \times \C^\times} (\pt)$-module. 
Fix a finite set of generators $F$ and map it to $\tilde F \subset \End^*_{\mc D^b_{C}(\mf g)}(K)$
by \eqref{eq:2.4}. That yields a finite rank $H^*_C (\pt)$-submodule $H^*_C (\pt) \tilde F$ of
$\End^*_{\mc D^b_{C}(\mf g)}(K)$. By parts (a) and (b)
\[
\hat H^*_C (\pt)_{\sigma,r} \tilde F \cong \hat H^*_{G \times \C^\times}(\pt)_{\sigma,r} F
\cong \hat H^*_C (\pt)_{\sigma,r} \End^*_{\mc D^b_{C}(\mf g)}(K)
\]
for all possible $(\sigma,r)$. Since localization is an exact functor, this implies that
\[
\End^*_{\mc D^b_{C}(\mf g)}(K) / H^*_C (\pt) \tilde F
\]
localizes to zero everywhere. Hence this quotient is zero and $\tilde F$ generates  
$\End^*_{\mc D^b_{C}(\mf g)}(K)$ as $H^*_C (\pt)$-module. The image of 
\[
H_C^* (\pt) \to \End^*_{\mc D^b_C (\mf g)}(K).
\]
is Noetherian because it is finitely generated and commutative. Hence $\End^*_{\mc D^b_C (\mf g)}(K)$
is Noetherian as well. 

By Theorem \ref{thm:1.2} the left hand side of part (b) only involves elements of even degrees $\geq 0$. 
Hence so does the right hand side, and its subalgebra $\End^*_{\mc D^b_C (\mf g)}(K)$.\\
(d) This can be shown in the same way as parts (b) and (c).
\end{proof}

\subsection{Localization on $\mf g$ and $\mf g_N$} \
\label{par:localization}

Having completed $\End^*_{\mc D^b_{G \times \C^\times}(\mf g)}(K)$ with respect to a central character,
we want to see how this affects the underlying variety $\mf g$. Let $T_{\sigma,r}$ be the smallest
algebraic torus in $G^\circ \times \C^\times$ whose Lie algebra contains $(\sigma,r)$. Then
$\mf g^{\sigma,r} := \mf g^{T_{\sigma,r}}$ is $C$-stable and
\begin{equation}\label{eq:2.10}
\begin{aligned}
\mf g^{\sigma,r} \; = \; \mf g^{\exp (\C (\sigma,r))} \; & = 
\{ X \in \mf g : e^{-2zr} \Ad (\exp (z \sigma)) X = X \; \forall z \in \C \} \\
& = \{ X \in \mf g : \Ad (\exp (z \sigma)) X = e^{2zr} X \; \forall z \in \C \} \\
& = \{ X \in \mf g : \mr{ad}(\sigma) X = 2 r X \} . 
\end{aligned}
\end{equation}
Notice that $\mf g^{\sigma,r}$ consists entirely of nilpotent elements (unless $r = 0$). We write
\begin{align*}
& \dot{\mf g} = \{ (X,gP) \in \mf g \times G/P : 
\Ad (g^{-1}) X \in \cC_v^M \oplus \mf t \oplus \mf u \} , \\
& \dot{\mf g}^{\sigma,r} = \dot{\mf g}^{T_{\sigma,r}} = 
\dot{\mf g} \cap \big( \mf g^{\sigma,r} \times (G / P)^{\exp (\C \sigma)} \big) .
\end{align*}
Consider the commutative diagram
\begin{equation}\label{eq:2.12}
\begin{array}{ccc}
\dot{\mf g}^{\sigma,r} & \xrightarrow{j_{\sigma,r}} & \dot{\mf g} \\
\downarrow \pr_1 & & \downarrow \pr_1 \\
\mf g^{\sigma,r} & \xrightarrow{j_{\sigma,r}} & \mf g
\end{array},
\end{equation}
where the vertical maps are inclusions. We define 
\[
K_{\sigma,r} = \pr_{1,!} j^*_{\sigma,r} (\dot{q\cE}) \qquad \in \mc D^b_C (\mf g^{\sigma,r}).
\]
Since \eqref{eq:2.12} is often not a pullback diagram, $K_{\sigma,r}$ need not be isomorphic to
$j^*_{\sigma,r} (K) = j^*_{\sigma,r} \pr_{1,!} (\dot{q\cE})$. Nevertheless $K_{\sigma,r}$ can be 
regarded as some kind of restriction of $K$ to $\mf g^{\sigma,r}$. According to \cite[\S 8.12]{Lus-Cusp2}
\begin{equation}\label{eq:2.21}
K_{\sigma,r} = \mr{pr}_{1,!} \IC_{Z_G (\sigma) \times \C^\times} 
\big( \mf g^{\sigma,r} \times (G / P)^{\exp (\C \sigma)},j^*_{\sigma,r} \dot{q\cE} \big) ,
\end{equation}
where now 
\begin{equation}\label{eq:2.13}
\mr{pr}_1 : \mf g^{\sigma,r} \times (G / P)^{\exp (\C \sigma)} \to \mf g^{\sigma,r} 
\quad \text{is proper.}
\end{equation}
As noted in  \cite[\S 5.3]{Lus-Cusp2} (where $K$ is called $B$ and pullbacks to 
$T_{\sigma,r}$-fixed subvarieties are indicated by a superscript tilde), this implies that 
$K_{\sigma,r}$ is a semisimple complex, that is, a direct sum of degree shifts of simple 
perverse sheaves on $\mf g^{\sigma,r}$.
Notice that for $(\sigma,r) = (0,0)$ we have 
\begin{equation}\label{eq:2.11}
\mf g^{0,0} = \mf g ,\quad \dot{\mf g}^{0,0} = \dot{\mf g} \quad \text{and} \quad K_{0,0} = K. 
\end{equation}
Thus the objects in Theorem \ref{thm:1.2} are special cases of their localized versions in this 
paragraph. Write 
\[
\mf g_N^{\sigma,r} = \mf g^{\sigma,r} \cap \mf g_N ,\qquad 
\dot{\mf g}_N^{\sigma,r} = \dot{\mf g}_N^{T_{\sigma,r}} =
\dot{\mf g}_N \cap \big( \mf g^{\sigma,r} \times (G / P)^{\exp (\C \sigma)} \big) .
\] 
Let $j_{N,\sigma,r} : \dot{\mf g}_N^{\sigma,r} \to \dot{\mf g}$ be the inclusion and define
\[
K_{N,\sigma,r} = (\pr_{1,N})_! j_{N,\sigma,r}^* (\dot{q\cE}_N) \qquad \in \mc D^b_C (\mf g_N^{\sigma,r}) .
\]
From the diagram
\[
\xymatrix{
\dot{\mf g}^{\sigma,r}_N \ar[r]^{j_{N,\sigma,r}} \ar[d]^{\mr{pr}_{1,N}} & 
\dot{\mf g}^{\sigma,r} \ar[d]^{\mr{pr}_1} \\
\mf g_N^{\sigma,r} \ar[r]^{j_{N,\sigma,r}} & \mf g^{\sigma,r}
}
\]
we see with base change \cite[Theorem 3.4.3]{BeLu} that $K_{N,\sigma,r}$ is the pullback of
$K_{\sigma,r}$ to $\mf g_N^{\sigma,r}$. We record that 
\[
\begin{array}{cccccccccc}
\mf g_N^{\sigma,r} & = & \mf g^{\sigma,r} , & \dot{\mf g}_N^{\sigma,r} & = &
\dot{\mf g}^{\sigma,r} , & K_{N,\sigma,r} & = & K_{\sigma,r} & \text{ for } r \neq 0 ,\\
\mf g_N^{\sigma,0} & = & Z_{\mf g} (\sigma)_N , & \dot{\mf g}_N^{\sigma,0} & = &
\dot{Z_{\mf g} (\sigma)}_N , & K_{N,\sigma,0} & = & K_{N,\sigma} & \text{ for } r = 0 .
\end{array}
\]
As both $K_{\sigma,r}$ (see above) and $K_{N,\sigma}$ are semisimple complexes 
\cite[Lemma 2.8]{SolSGHA}, $K_{N,\sigma,r}$ is always a semisimple object of 
$\mc D^b_C (\mf g_N^{\sigma,r})$.

\begin{prop}\label{prop:2.4}
\enuma{
\item There exists a natural isomorphism of $\hat{H}^*_C (\pt)_{\sigma,r}$-algebras
\[
\hat{H}^*_C (\pt)_{\sigma,r} \underset{H_C^* (\pt)}{\otimes} 
\End^*_{\mc D^b_{C}(\mf g^{\sigma,r})}(K_{\sigma,r}) \longrightarrow 
\hat{H}^*_C (\pt)_{\sigma,r} \underset{H_C^* (\pt)}{\otimes} \End^*_{\mc D^b_{C}(\mf g)}(K) .
\]
\item The algebras in part (a) are naturally isomorphic with
\[
\hat{H}^*_C (\pt)_{\sigma,r} \underset{H_C^* (\pt)}{\otimes} 
\End^*_{\mc D^b_{C}(\mf g_N^{\sigma,r})}(K_{N,\sigma,r}) \cong
\hat{H}^*_C (\pt)_{\sigma,r} \underset{H_C^* (\pt)}{\otimes} \End^*_{\mc D^b_{C}(\mf g_N)}(K_N) .
\]
\item The graded algebras $\End^*_{\mc D^b_{C}(\mf g^{\sigma,r})}(K_{\sigma,r})$ and 
$\End^*_{\mc D^b_{C}(\mf g_N^{\sigma,r})}(K_{N,\sigma,r})$ are Noetherian and only have terms
in even degrees.
}
\end{prop}
\begin{proof}
(a) In \cite[\S 4.9--4.10]{Lus-Cusp2}, which is applicable by Appendix \ref{sec:B}, this was 
proven under the assumption that $G$ (and hence $C$) is connected. Explicitly, there exists 
an isomorphism of $\hat{H}^*_{C^\circ} (\pt)_{\sigma,r}$-algebras
\[
\hat{H}^*_{C^\circ} (\pt)_{\sigma,r} \underset{H_{C^\circ}^* 
(\pt)}{\otimes} \End^*_{\mc D^b_{C^\circ}(\mf g^{\sigma,r})}(K_{\sigma,r}) \longrightarrow
\hat{H}^*_{C^\circ} (\pt)_{\sigma,r} \underset{H_{C^\circ}^* (\pt)}{\otimes} 
\End^*_{\mc D^b_{C^\circ}(\mf g)}(K) . 
\]
With the same argument as in the proof of Proposition \ref{prop:2.3}.b, we can take 
$C/C^\circ$-invariants on both side, and that replaces all occurences of $C^\circ$ by $C$.\\
(b) By Proposition \ref{prop:2.3}.c and Theorem \ref{thm:1.2} there is a natural isomorphism
\[
\hat{H}^*_C (\pt)_{\sigma,r} \underset{H_C^* (\pt)}{\otimes} \End^*_{\mc D^b_{C}(\mf g)}(K) \cong
\hat{H}^*_C (\pt)_{\sigma,r} \underset{H_C^* (\pt)}{\otimes} \End^*_{\mc D^b_{C}(\mf g_N)}(K_N) .
\] 
When $r \neq 0$, the left hand sides of parts (a) and (b) are the same. 
When $r = 0$, we assume (as we may by Lemma \ref{lem:2.2}) that $\sigma \in \mf t$.
In the notations from \cite[\S 2.3]{SolSGHA} we have
\[
\mf g^{\sigma,0} = Z_{\mf g}(\sigma), \quad \dot{\mf g}^{\sigma,0} = \dot{\mf g}^\sigma \quad
\text{and} \quad K_{\sigma,0} = K_{\sigma} .
\] 
Then \cite[Proposition 2.10]{SolSGHA} provides the final isomorphism 
\[
\hat{H}^*_C (\pt)_{\sigma,r} \underset{H_C^* (\pt)}{\otimes} 
\End^*_{\mc D^b_{C}(\mf g^{\sigma,0})}(K_{\sigma,0}) \to
\hat{H}^*_C (\pt)_{\sigma,r} \underset{H_C^* (\pt)}{\otimes} 
\End^*_{\mc D^b_{C}(\mf g_N^{\sigma,0})}(K_{N,\sigma,0}) .
\]
(c) This can be proven like in Proposition \ref{prop:2.3}.c.
\end{proof}

We let $W_{q\cE}$ act on $\Sigma_v (\mf t \oplus \C)$ by decreeing that $\Sigma_v$ is 
$W_{q\cE}$-equivariant. For $(\sigma,r) \in \Ad (G)\Sigma_v (\mf t \oplus \C)$, let $Z_{\sigma,r}$ 
be the maximal ideal of 
\[
\mc O (\Sigma_v (\mf t \oplus \C)/W_{q\cE}) = \mc O (\mf t / W_{q\cE} \times \C (\sigma_v,1)) 
\cong \mc O (\mf t / W_{q\cE} \times \C) \subset Z(\mh H (G,M,q\cE)) 
\] 
determined by $(\sigma,r)$ via Lemma \ref{lem:2.2}.
Every finite length module $V$ of $\mh H (G,M,q\cE)$ can be decomposed as
\begin{align*}
& V = \bigoplus\nolimits_{(\sigma,r) \in \Sigma_v (\mf t \oplus \C) / W_{q\cE}} V_{\sigma,r} \;, \\
& V_{\sigma,r} = \{ v \in V : Z_{\sigma,r}^n \cdot v = 0 \text{ for some } n \in \N \}. 
\end{align*}
Hence the category of finite length left modules is a direct sum
\begin{equation}\label{eq:2.8}
\Mod_{\mr{fl}} (\mh H (G,M,q\cE)) = \bigoplus\nolimits_{(\sigma,r) \in 
\Sigma_v (\mf t \oplus \C) / W_{q\cE}} \Modf{\sigma,r} (\mh H (G,M,q\cE)) .
\end{equation}
The same holds for right modules:
\begin{equation}\label{eq:2.9}
\mh H (G,M,q\cE) - \Mod_{\mr{fl}} = \bigoplus\nolimits_{(\sigma,r) \in 
\Sigma_v (\mf t \oplus \C) / W_{q\cE}} \mh H (G,M,q\cE) - \Modf{\sigma,r} .
\end{equation}
Let $\hat Z (\mh H (G,M,q\cE))_{\sigma,r}$ be the formal completion of $Z(\mh H (G,M,q\cE))$ 
with respect to $Z_{\sigma,r}$. Then $\Modf{\sigma,r} (\mh H (G,M,q\cE))$ can be identified with
the category of finite length left modules, continuous with respect to the adic topology, of the 
completed algebra
\[
\hat Z (\mh H (G,M,q\cE))_{\sigma,r} \underset{Z (\mh H (G,M,q\cE))}{\otimes} \mh H (G,M,q\cE) .
\]
We use a similar notation for the algebras $\End^*_{\mc D^b_{G \times \C^\times}(\mf g)}(K)$ and \\
$\End^*_{\mc D^b_{Z_G (\sigma) \times \C^\times}(\mf g^{\sigma,r} )}(K_{\sigma,r})$, 
with respect to the maximal ideals determined by $(\sigma,r)$ in $H_{G \times \C^\times}^* (\pt)$ 
and in $H^*_{Z_G (\sigma) \times \C^\times}(\pt)$.

\begin{thm}\label{thm:2.5}
\enuma{
\item There are natural algebra isomorphisms
\begin{align*}
& \hat Z (\mh H (G,M,q\cE))_{\sigma,r} \underset{Z (\mh H (G,M,q\cE))}{\otimes}
\mh H (G,M,q\cE) \; \isom \\
& \hat{H}^*_{G \times \C^\times} (\pt)_{\sigma,r} \underset{H_{G \times \C^\times}^* 
(\pt)}{\otimes} \End^*_{\mc D^b_{G \times \C^\times}(\mf g)}(K) \; \isom \\
& \hat{H}^*_{Z_G (\sigma) \times \C^\times} (\pt)_{\sigma,r} \underset{H_{Z_G (\sigma) \times \C^\times}^* 
(\pt)} {\otimes} \End^*_{\mc D^b_{Z_G (\sigma) \times \C^\times}(\mf g^{\sigma,r})} (K_{\sigma,r}) .
\end{align*}
\item Part (a) induces equivalences of categories
\begin{align*}
\Modf{\sigma ,r} (\mh H (G,M,q\cE)) & \cong 
\Modf{\sigma ,r} \big( \End^*_{\mc D^b_{G \times \C^\times}(\mf g)}(K) \big) \\  
& \cong \Modf{\sigma ,r} \big( \End^*_{\mc D^b_{Z_G (\sigma) \times \C^\times}
(\mf g^{\sigma,r})}(K_{\sigma,r}) \big) ,
\end{align*}
and analogously with right modules.
\item Parts (a) and (b) also hold with $(\mf g_N, K_N)$ instead of $(\mf g, K)$.
} 
\end{thm}
\begin{proof}
(a) is a consequence of Theorem \ref{thm:1.2} and Propositions \ref{prop:2.3}, \ref{prop:2.4}.\\
(b) follows directly from (a).\\
(c) The proof is completely analogous to that of parts (a) and (b).
\end{proof}

We point out that the data $(\sigma,r)$ in Theorem \ref{thm:2.5} can be scaled by an arbitrary
$z \in \C^\times$. Namely, 
\begin{equation}\label{eq:2.18}
\mf g^{z \sigma, z r} = \mf g^{\sigma,r} ,\quad \dot{\mf g}^{z \sigma, z r} = 
\dot{\mf g}^{\sigma,r} , \quad K_{z \sigma,z r} = K_{\sigma,r}
\end{equation}
and similarly with subscripts $N$.

\subsection{Twisted graded Hecke algebras with a fixed $r$} \
\label{par:fixed}

So far we mainly considered twisted graded Hecke algebras with a formal variable $\mb r$.
Often we localized $\mb r$ at a complex number $r$, but still we allowed modules
on which $\mb r$ did not act as a scalar. In the Hecke algebras that arise from reductive
$p$-adic groups, $\mb r$ is always specialized to some $r \in \R$, see \cite{SolEnd}.
That prompts us to find versions of our previous results for such algebras.

Fix $r \in \C$ and write
\[
\mh H (G,M,q\cE,r) = \mh H (G,M,q\cE) / (\mb r - r) = \mh H (\mf t, W_{q\cE}, c r, \natural_{q\cE}) .
\]
The centre of $\mh H (G,M,q\cE,r)$ contains $\mc O (\mf t / W_{q\cE})$. It will be convenient to
identify $\mf t$ with $\mf t_r = \mf t + r \sigma_v$ via $\Sigma_v$, and to identify 
$\mf t / W_{q\cE}$ with $\mf t_r / W_{q\cE}$. This enables us to localize $\mh H (G,M,q\cE,r)$ at
$W_{q\cE} \sigma \in \mf t_r / W_{q\cE}$, which is consistent with the previous paragraph.
The irreducible and standard modules of $\mh H (G,M,q\cE,r)$ have already been classified in 
\cite{AMS2} they come from $\mh H (G,M,q\cE)$ by imposing that
$\mb r$ acts as $r$. However, the Ext-groups of two $\mh H (G,M,q\cE,r)$-modules are usually not
isomorphic to their Ext-groups as $\mh H (G,M,q\cE)$-modules.
 
As $\C [\mb r] \cong H^*_{\C^\times}(\pt)$ comes in from the $\C^\times$-actions on our varieties 
and sheaves, it is natural to try to replace $Z_G (\sigma) \times \C^\times$-equivariant sheaves by 
$Z_G (\sigma)$-equivariant sheaves in Paragraph \ref{par:localization}. However, the 
$\C^\times$-actions are there for a reason. Without them, \cite{Lus-Cusp1} would just give
\[
\End^*_{\mc D^b_G (\mf g)}(K) \cong \mc O (\mf t) \rtimes \C [W_{q\cE},\natural_{q\cE}] \cong
\mh H (G,M,q\cE) / (\mb r) ,
\]
and from there one would never get any $r$ in the picture. Therefore we proceed more subtly, first
we formally complete with respect to $(\sigma,r)$ and only then we forget the $\C^\times$-actions.
For $\sigma \in \mf t_r$, Theorem \ref{thm:2.5}.a implies
\begin{equation}\label{eq:8.3}
\begin{aligned}
& \hat{Z} (\mh H (G,M,q\cE,r))_\sigma \underset{Z (\mh H (G,M,q\cE,r))}{\otimes} \mh H (G,M,q\cE,r) \; \cong \\
& \hat{Z} (\mh H (G,M,q\cE))_{\sigma,r} / (\mb r - r) 
\underset{Z (\mh H (G,M,q\cE))}{\otimes} \mh H (G,M,q\cE) \; \cong \\
& \hat{H}^*_{G \times \C^\times} (\pt)_{\sigma,r} / (\mb r - r) 
\underset{H^*_{G \times \C^\times}(\pt)}{\otimes} \End^*_{\mc D^b_{G \times \C^\times} (\mf g)} (K) \; \cong \\
& \hat{H}^*_{Z_G (\sigma) \times \C^\times} (\pt)_{\sigma,r} / (\mb r - r) \underset{H^*_{Z_G (\sigma) \times 
\C^\times}(\pt)}{\otimes} \End^*_{\mc D^b_{Z_G (\sigma) \times \C^\times} (\mf g^{\sigma,r})} (K_{\sigma,r}) .
\end{aligned}
\end{equation}
It is much easier to analyse \eqref{eq:8.3} when $r = 0$, so we settle that case first.

\begin{lem}\label{lem:8.1}
For $\sigma \in \mf t$ there are natural algebra isomorphisms
\begin{align*}
& \hat{Z} (\mh H (G,M,q\cE,0))_\sigma \underset{Z (\mh H (G,M,q\cE,0))}{\otimes} \mh H (G,M,q\cE,0) \; \cong \\
& \hat{H}^*_{Z_G (\sigma)} (\pt)_{\sigma} \underset{H^*_{Z_G (\sigma)}(\pt)}{\otimes} 
\End^*_{\mc D^b_{Z_G (\sigma)} (\mf g^{\sigma,0})} (K_{\sigma,0}) \cong \\
& \hat{H}^*_{Z_G (\sigma)} (\pt)_{\sigma} \underset{H^*_{Z_G (\sigma)}(\pt)}{\otimes} 
\End^*_{\mc D^b_{Z_G (\sigma)} (\mf g_N^{\sigma,0})} (K_{N,\sigma,0}) .
\end{align*}
\end{lem}
\begin{proof}
The final line of \eqref{eq:8.3} simplifies because
\begin{equation}\label{eq:8.4}
\hat{H}^*_{Z_G (\sigma) \times \C^\times} (\pt)_{\sigma,0} / (\mb r) \cong \hat{H}^*_{Z_G (\sigma)} 
(\pt)_\sigma \otimes \hat{H}^*_{\C^\times} (\pt)_0 / (\mb r) \cong \hat{H}^*_{Z_G (\sigma)} (\pt)_{\sigma} 
\end{equation}
as $H^*_{Z_G (\sigma) \times \C^\times}(\pt)$-modules. Further, by \cite[\S 4.11]{Lus-Cusp2} there is a 
natural isomorphism 
\begin{multline}\label{eq:8.5}
\hat{H}^*_{Z_G^\circ (\sigma)} (\pt)_{\sigma} \underset{H^*_{Z_G^\circ (\sigma) \times \C^\times}(\pt)}{\otimes}
\End^*_{\mc D^b_{Z_G^\circ (\sigma) \times \C^\times} (\mf g^{\sigma,0})} (K_{\sigma,0}) \cong \\
\hat{H}^*_{Z_G^\circ (\sigma)} (\pt)_{\sigma} \underset{H^*_{Z_G^\circ (\sigma)}(\pt)}{\otimes}
\End^*_{\mc D^b_{Z_G^\circ (\sigma)} (\mf g^{\sigma,0})} (K_{\sigma,0}) .
\end{multline}
Taking $\pi_0 (Z_G (\sigma))$-invariants, as in the proof of Proposition \ref{prop:2.3}, we obtain
the analogue of \eqref{eq:8.5} with $Z_G (\sigma)$ instead of $Z_G^\circ (\sigma)$. Combining that with
\eqref{eq:8.3} and \eqref{eq:8.4} proves the first isomorphism.

The isomorphism between the first and third terms in the statement can be shown in the same way, 
starting from part (c) instead of part (a) of Theorem \ref{thm:2.5}. 
\end{proof}

Next we consider a nonzero $r$ and we fix $\sigma \in \mf t_r$. Notice that $\mf t_r \subset \mf m$,
so $\sigma$ commutes with $\mf t$ and with $T = \exp (\mf t)$. 
Let $T'$ be a maximal torus of $Z_G^\circ (\sigma)$ containing $T$.

\begin{lem}\label{lem:8.2}
There is a natural algebra isomorphism
\begin{multline*}
\hat{H}^*_{Z_G (\sigma) \times \C^\times} (\pt)_{\sigma,r} \underset{H^*_{Z_G (\sigma) \times \C^\times}
(\pt)}{\otimes} \End^*_{\mc D^b_{Z_G (\sigma) \times \C^\times} (\mf g^{\sigma,r})} (K_{\sigma,r}) \; \cong \\
\C [[\mb r -r]] \otimes_\C \hat{H}^*_{Z_G (\sigma)} (\pt)_{\sigma} \underset{H^*_{Z_G (\sigma)}(\pt)}{\otimes}
\End^*_{\mc D^b_{Z_G (\sigma)} (\mf g^{\sigma,r})} (K_{\sigma,r}) .
\end{multline*}
\end{lem}
\begin{proof}
From \cite[\S 4.11]{Lus-Cusp2} and Proposition \ref{prop:B.1} we get
\begin{multline}\label{eq:8.6}
\hat{H}^*_{Z_G^\circ (\sigma) \times \C^\times} (\pt)_{\sigma,r} \underset{H^*_{Z_G^\circ (\sigma) \times 
\C^\times}(\pt)}{\otimes} \End^*_{\mc D^b_{Z_G^\circ (\sigma) \times \C^\times} (\mf g^{\sigma,r})} 
(K_{\sigma,r}) \; \cong \\
\hat{H}^*_{T' \times \C^\times} (\pt)_{\sigma,r} \underset{H^*_{T' \times \C^\times}(\pt)}{\otimes} 
\End^*_{\mc D^b_{T' \times \C^\times} (\mf g^{\sigma,r})} (K_{\sigma,r}) .
\end{multline}
The subgroup $\exp (\C (\sigma,r)) \subset T' \times \C^\times$ fixes $\mf g^{\sigma,r}$ pointwise
and projects onto $\C^\times$ because $r \neq 0$. Hence 
\begin{equation}\label{eq:8.8}
T' \times \C^\times = T' \times \exp (\C (\sigma,r)) \quad\text{in}\quad G^\circ \times \C^\times .
\end{equation}
By connectedness $\exp (\C (\sigma,r))$ also acts trivially on $K_{\sigma,r}$,
so we can further decompose according to \eqref{eq:8.8}:
\begin{equation}\label{eq:8.9}
\begin{aligned}
\End^*_{\mc D^b_{T' \times \C^\times}(\mf g^{\sigma,r})} (K_{\sigma,r})  
& \cong \End^*_{\mc D^b_{\exp (\C (\sigma,r))}(\pt)} (\C) \otimes_\C 
\End^*_{\mc D^b_{T'}(\mf g^{\sigma,r})} (K_{\sigma,r}) \\
& \cong H^*_{\exp (\C (\sigma,r))} (\pt) \otimes_\C \End^*_{\mc D^b_{T'}(\mf g^{\sigma,r})} (K_{\sigma,r}) \\ 
& \cong \C [\mb r - r] \otimes_\C \End^*_{\mc D^b_{T'}(\mf g^{\sigma,r})} (K_{\sigma,r}) .
\end{aligned}
\end{equation}
Then \eqref{eq:8.6} and its analogue without $\C^\times$ yield 
\begin{multline}\label{eq:8.11}
\begin{aligned}
& \hat{H}^*_{Z_G^\circ (\sigma) \times \C^\times} (\pt)_{\sigma,r} \underset{H^*_{Z_G^\circ (\sigma) \times 
\C^\times}(\pt)}{\otimes} \End^*_{\mc D^b_{Z_G^\circ (\sigma) \times \C^\times} (\mf g^{\sigma,r})} 
(K_{\sigma,r}) \; \cong \\
& \C [[\mb r - r]] \otimes \hat{H}^*_{T'} (\pt)_{\sigma} \underset{H^*_{T'} (\pt)}{\otimes} 
\End^*_{\mc D^b_{T'}(\mf g^{\sigma,r})}(K_{\sigma,r}) \; \cong \\
& \C [[\mb r - r]] \otimes \hat{H}^*_{Z_G^\circ (\sigma)} (\pt)_{\sigma} \underset{H^*_{Z_G^\circ (\sigma)}(\pt)}
{\otimes} \End^*_{\mc D^b_{Z_G^\circ (\sigma)} (\mf g^{\sigma,r})} (K_{\sigma,r}) .
\end{aligned}
\end{multline}
As in the proof of Proposition \ref{prop:2.3}.b, taking $\pi_0 (Z_G (\sigma))$-invariants in 
\eqref{eq:8.11} replaces $Z_G^\circ (\sigma)$ by $Z_G (\sigma)$.
\end{proof}

Now we can prove our desired variation on Theorem \ref{thm:2.5}. 

\begin{thm}\label{thm:8.3}
Let $r \in \C$ and $\sigma \in \mf t_r$. 
\enuma{
\item There exists a natural algebra isomorphism
\begin{multline*}
\hat{Z} (\mh H (G,M,q\cE,r))_\sigma \underset{Z (\mh H (G,M,q\cE,r))}{\otimes} \mh H (G,M,q\cE,r) 
\; \cong \\
\hat{H}^*_{Z_G (\sigma)} (\pt)_{\sigma} \underset{H^*_{Z_G (\sigma)}(\pt)}{\otimes} 
\End^*_{\mc D^b_{Z_G (\sigma)} (\mf g^{\sigma,r})} (K_{\sigma,r}) .
\end{multline*}
\item This induces an equivalence of categories
\[
\Modf{\sigma} \big( \mh H (G,M,q\cE,r) \big) \cong \Modf{\sigma} \big( 
\End^*_{\mc D^b_{Z_G (\sigma)} (\mf g^{\sigma,r})} (K_{\sigma,r}) \big) .
\]
\item Parts (a) and (b) also hold with $(\mf g_N^{\sigma,r}, K_{N,\sigma,r})$ instead of
$(\mf g^{\sigma,r},K_{\sigma,r})$.
}
\end{thm}
\begin{proof}
(b) follows directly from (a).\\
(a,c) For $r=0$ this is Lemma \ref{lem:8.1}, so we may assume $r \neq 0$. Then the subscripts $N$
do not change anything. By \eqref{eq:8.3} and Lemma \ref{lem:8.2} 
\begin{align*}
& \hat{Z} (\mh H (G,M,q\cE,r))_\sigma \underset{Z (\mh H (G,M,q\cE,r))}{\otimes} 
\mh H (G,M,q\cE,r) \; \cong \\
& \big( \C [[\mb r -r]] \otimes_\C \hat{H}^*_{Z_G (\sigma)} (\pt)_{\sigma} \big) / (\mb r - r) 
\underset{H^*_{Z_G (\sigma)}(\pt)}{\otimes} \End^*_{\mc D^b_{Z_G (\sigma)} (\mf g^{\sigma,r})} 
(K_{\sigma,r}) \; \cong \\
& \hat{H}^*_{Z_G (\sigma)} (\pt)_{\sigma} \underset{H^*_{Z_G (\sigma)}(\pt)}{\otimes} 
\End^*_{\mc D^b_{Z_G (\sigma)} (\mf g^{\sigma,r})} (K_{\sigma,r}) . \qedhere
\end{align*}
\end{proof}

\section{Standard modules of twisted graded Hecke algebras}
\label{sec:standard}

In \cite{Lus-Cusp1,AMS2} standard (left) modules for $\mh H (G,M,q\cE)$ were studied.
We will quickly recall their construction and then we relate these standard modules to the 
previous section. Recall that Condition \ref{cond:2.1} is in force. In this section the
data $(G,M,q\cE)$ are fixed, and we often abbreviate $\mh H = \mh H (G,M,q\cE)$.

Let $y \in \mf g$ be nilpotent and define 
\[
\mc P_y = \{ gP \in G/P : \Ad (g^{-1}) y \in \cC_v^M + \mf u \} .
\]
The group
\begin{equation}\label{eq:4.1}
Z_{G \times \C^\times} (y) = \{ (g_1,\lambda) \in G \times \C^\times : \Ad (g_1) y = \lambda^2 y \}
\end{equation}
acts on $\mc P_y$ by $(g_1,\lambda) \cdot gP = g_1 g P$. This puts $\mc P_y$ in 
$Z_{G \times \C^\times} (y)$-equivariant bijection with $\{y\} \times \mc P_y \subset \dot{\mf g}$.
The local system $\dot{q\cE}$ on $\dot{\mf g}$ restricts to a local system on $\{y\} \times \mc P_y 
\cong \mc P_y$, still called $\dot{q\cE}$. 
The action of $\C [W_{q\cE} ,\natural_{q\cE}]$ on $K$ from Theorem \ref{thm:1.2} induces an action on
\begin{equation}\label{eq:4.6}
H_*^{G \times \C^\times} (\mf g,K) \cong H_*^{G \times \C^\times}(\dot{\mf g}, \dot{q\cE})
\end{equation}
From \cite[(2.4)]{SolSGHA} and the product in equivariant cohomology, we get an action of 
$\mc O (\mf t \oplus \C)$ on \eqref{eq:4.6}. These can be pulled back to actions on 
\[
H_*^\Zy (\mc P_y,\dot{q\cE}) ,
\]
making that vector space into a graded left module over $\mh H (G,M,q\cE)$ and over $H^*_\Zy (\pt)$.
Further, $Z_{G \times \C^\times}(y)$ acts naturally on $H_*^\Zy (\mc P_y,\dot{q\cE})$ and on 
$H^*_\Zy (\pt)$, and those actions factor through the component group $\pi_0 (Z_{G \times \C^\times}(y))$.

\begin{thm} \textup{(see \cite[Theorem 8.13]{Lus-Cusp1} and \cite[Theorem 3.2 and \S 4]{AMS2})} 
\label{thm:4.1}
\enuma{
\item The actions of $\mh H$ and $H^*_\Zy (\pt)$ on $H_*^\Zy (\mc P_y,\dot{q\cE})$ commute.
\item As $H^*_\Zy (\pt)$-module, $H_*^\Zy (\mc P_y,\dot{q\cE})$ is finitely generated and free.
\item The action of $\pi_0 (Z_{G \times \C^\times}(y))$ on $H_*^\Zy (\mc P_y,\dot{q\cE})$ commutes
with the action of $\mh H$ and is semilinear with respect to $H^*_\Zy (y)$.
}
\end{thm}
\begin{proof}
Comparing with the references, it only remains to see that in part (b) the module is free. 
This part ultimately relies on \cite[Proposition 7.2]{Lus-Cusp1}, where it is proven that the module
is finitely generated and projective. However, that argument actually shows that the module is free.
\end{proof}
 
Recall from \cite[\S 1.11]{Lus-Cusp1} that 
\begin{multline}\label{eq:4.2}
H^*_\Zy (\pt) \text{ is the ring of invariant polynomials on}\\
\text{the maximal reductive quotient of } \mr{Lie} (\Zy) \text{, with doubled degrees.}
\end{multline} 
The characters of $H^*_\Zy (\pt)$ are parametrized by the
semisimple adjoint orbits in the reductive quotient Lie algebra. 

We let $\mf g \oplus \C$ act on $\mf g$ by 
\[
(\sigma,r) \cdot X = [\sigma,X] - 2rX ,
\]
that is the derivative of the $G \times \C^\times$-action. Then we can write
\begin{equation}\label{eq:4.3}
\mr{Lie} (Z_{G \times \C^\times}(y)) = \{ (\sigma,r) \in \mf g \oplus \C : [\sigma,y] = 2 ry \}
= Z_{\mf g \oplus \C}(y) .
\end{equation}
Thus every semisimple $(\sigma,r) \in Z_{\mf g \oplus \C}(y)$ defines a unique character of 
$H^*_\Zy (\pt)$, which we denote $\C_{\sigma,r}$. This gives us a family of $\mh H$-modules
\[
E_{y,\sigma,r} := \C_{\sigma,r} \underset{H^*_\Zy (\pt)}{\otimes} H_*^\Zy (\mc P_y, \dot{q\cE})
\qquad \text{for semisimple } (\sigma,r) \in Z_{\mf g \oplus \C}(y) .
\] 
It is known from \cite[Proposition 1.4]{AMS3} that (when Condition \ref{cond:2.1} holds) 
$E_{y,\sigma,r}$ admits the central character $((\Ad (G) \sigma - r \sigma_v) \cap \mf t,r)$. 
Via Lemma \ref{lem:2.2} this corresponds to $\Ad (G) (\sigma,r) \cap \Sigma_v (\mf t \oplus \C)$.
Let 
\[
C_y = Z_C (y) = Z_{G \times \C^\times} (y,\sigma,r)
\] 
be the intersection of $Z_{G \times \C^\times}(y)$ from \eqref{eq:4.1} and 
$C = Z_{G \times \C^\times}(\sigma,r)$ (with respect to the adjoint action). 
The component group $\pi_0 (C_y)$ acts naturally on $E_{y,\sigma,r}$ by $\mh H$-intertwiners. 
For $\rho \in \Irr (\pi_0 (C_y))$ we form the $\mh H$-module
\begin{equation}\label{eq:4.7}
E_{y,\sigma,r,\rho} = \Hom_{\pi_0 (C_y)}(\rho, E_{y,\sigma,r}) .
\end{equation}
Choose an algebraic homomorphism $\gamma_y : SL_2 (\C) \to G^\circ$ with 
d$\gamma_y \matje{0}{1}{0}{0} = y$. It is often convenient to involve the element
\begin{equation}\label{eq:4.4}
\sigma_0 := \sigma + \textup{d}\gamma_y \matje{-r}{0}{0}{r} \in Z_{\mf g}(y) .
\end{equation}
For instance, by \cite[Lemma 3.6.a]{AMS2} there are natural isomorphisms 
\begin{equation}\label{eq:4.5}
\pi_0 (C_y) \cong \pi_0 (Z_G (y,\sigma)) \cong \pi_0 (Z_G (y,\sigma_0))
\end{equation}
It was shown in \cite[Proposition 3.7 and \S 4]{AMS2} that $E_{y,\sigma,r\rho}$ is nonzero
if and only if the cuspidal quasi-support $q\Psi_{Z_G (\sigma_0)}(y,\rho)$, for the
group $Z_G (\sigma_0)$ and with $\rho$ considered as representation of $\pi_0 (Z_G (y,\sigma_0))$
via \eqref{eq:4.5}, is $G$-conjugate to $(M,\cC_v^M,q\cE)$. Equivalent conditions will be 
described in Proposition \ref{prop:2.6}. For such $\rho$ we call 
$E_{y,\sigma,r,\rho}$ a standard (geometric) $\mh H$-module.

\begin{thm} \label{thm:4.2} \textup{\cite[Theorem 1.6]{AMS3}} 
\enuma{
\item For $r \in \C^\times$, every standard $\mh H$-module $E_{y,\sigma,r,\rho}$ has a
unique irreducible quotient $M_{y,\sigma,r,\rho}$.
\item For $r = 0$, the standard module $E_{y,\sigma,0,\rho}$ has a distinguished 
irreducible quotient, called $M_{y,\sigma,0,\rho}$.
\item For any $r \in \C$, the correspondence $M_{y,\sigma,r,\rho} \longleftrightarrow 
(y,\sigma,\rho)$ provides a bijection between 
\[
\Irr_r (\mh H) = \Irr (\mh H / (\mb r - r))
\] 
and the $G$-association classes of triples $(y,\sigma,\rho)$ as in \eqref{eq:4.3}--\eqref{eq:4.7}.
}
\end{thm}

\subsection{Relation with automorphisms of $\mh H$} \
\label{par:involutions}

We investigate how the standard $\mh H$-modules change under composition with automorphisms
of $\mh H = \mh H (G,M,q\cE)$. This will help us to make the parametrization of 
$\Irr_r (\mh H)$ from Theorem \ref{thm:4.2} with $r \in \R_{>0}$ compatible with the 
analytic properties temperedness and (essentially) discrete series. When $G$ is connected, that is 
worked out in \cite{Lus-Cusp3}. Unfortunately the outcome is not exactly what we want, 
it rather produces ``anti-tempered" representations where we would like temperedness.

For any $z \in \C^\times$, we have the scaling by degree automorphism
\[
\zeta_z : \mh H \to \mh H ,
\]
which multiplies every element in degree $2n$ by $z^n$. 
In addition to \eqref{eq:2.18}, we note that $\mc P_{zy} = \mc P_y$ and
\[
Z_{G \times \C^\times}(zy) = Z_{G \times \C^\times}(y).
\]

\begin{prop}\label{prop:5.1}
Let $E_{y,\sigma,r,\rho}$ be a geometric standard $\mh H$-module, as in Theorem 
\ref{thm:4.2}. There are canonical isomorphisms of $\mh H$-modules
\enuma{
\item $\zeta_z^* E_{y,\sigma,r} \cong E_{y,z\sigma,zr}$,
\item $\zeta_z^* E_{y,\sigma,r,\rho} \cong E_{y,z\sigma,zr,\rho}$,
\item $\zeta_z^* M_{y,\sigma,r,\rho} \cong M_{y,z\sigma,zr,\rho}$.
}
\end{prop} 
\begin{proof}
(a) By \cite[Proposition 8.6]{Lus-Cusp1} and \cite[Theorem 3.2 and (91)]{AMS2}, the graded
$\mh H$-module $H_*^\Zy (\mc P_y, \dot{q\cE})$ only has terms in even degrees. 
We define the linear bijection 
\begin{align*}
& \lambda_z : H_*^\Zy (\mc P_y, \dot{q\cE}) \to H_*^\Zy (\mc P_y, \dot{q\cE}) \\
& \lambda_z (x) = z^d x \quad \text{if} \deg (x) = 2d.
\end{align*}
This can be regarded as an isomorphism of $\mh H$-modules
\begin{equation}\label{eq:5.4}
(\zeta_z)^* H_*^\Zy (\mc P_y, \dot{q\cE}) \xrightarrow{\lambda_z} H_*^\Zy (\mc P_y, \dot{q\cE}) .
\end{equation}
Analogously we define the algebra automorphism
\[
\lambda'_z : H^*_\Zy (\pt) \to H^*_\Zy (\pt),
\]
and we note that 
\[
\lambda'_z (\ker (\mr{ev}_{\sigma,r}))  = \ker (\mr{ev}_{z\sigma,zr}).
\]
As $H_*^\Zy (\mc P_y, \dot{q\cE})$ is also a graded $H^*_\Zy (\pt)$-module:
\[
\lambda_z (m \cdot x) = \lambda'_z (m) \lambda (x) \quad \forall m \in H^*_\Zy (\pt),
x \in H_*^\Zy (\mc P_y, \dot{q\cE}) .
\]
It follows that \eqref{eq:5.4} induces an isomorphism of $\mh H$-modules
\begin{multline}\label{eq:5.2}
(\zeta_z)^* E_{y,\sigma,r} = (\zeta_z)^* \big( H_*^\Zy (\mc P_y, \dot{q\cE}) /
\ker (\mr{ev}_{\sigma,r}) H_*^\Zy (\mc P_y, \dot{q\cE}) \big) \xrightarrow{\lambda_z} \\
\big( H_*^\Zy (\mc P_y, \dot{q\cE}) / \ker (\mr{ev}_{z\sigma,zr}) H_*^\Zy (\mc P_y, \dot{q\cE})
\big) = E_{y,z\sigma,zr} .
\end{multline}
(b) The group $C_y$ acts naturally on $H_*^\Zy (\mc P_y, \dot{q\cE})$ and that induces the action of 
$\pi_0 (C_y) \cong \pi_0 (Z_G (y,\sigma))$ on $E_{y,\sigma,r}$ and on $E_{y,z\sigma,zr}$. 
The $C_y$-action on $H_*^\Zy (\mc P_y, \dot{q\cE})$ preserves the degrees, so it commutes with 
$\lambda_z$. It follows that \eqref{eq:5.2} is $C_y$-equivariant as well. In particular, for any 
$\rho \in \Irr (\pi_0 (C_y) )$, \eqref{eq:5.2} induces isomorphisms of $\mh H$-modules
\begin{equation}\label{eq:5.8}
(\zeta_z)^* E_{y,\sigma,r,\rho} = \Hom_{C_y} (\rho, (\zeta_z)^* E_{y,\sigma,r}) 
\xrightarrow{\lambda_z} \Hom_{C_y} (\rho, E_{y,z\sigma,zr}) = E_{y,z\sigma,zr,\rho} . 
\end{equation} 
(c) When $r \neq 0$, \eqref{eq:5.8} sends the unique irreducible quotient $(\zeta_z)^* M_{y,
\sigma,r,\rho}$ on the left to the unique irreducible quotient $M_{y,z\sigma,zr,\rho}$ on the right. 

When $r = 0$, the distinguished irreducible summand $(\zeta_z)^* M_{y,\sigma,0,\rho}$ of
$(\zeta_z)^* E_{y,\sigma,0,\rho}$ is in the component of $E_{y,\sigma,0,\rho} \cong 
H_* (\mc P_y,\dot{q\cE})$ in one particular homological degree \cite[Lemma 3.10 and 
Theorem 3.20]{AMS2}. As \eqref{eq:5.2} preserves these homological degrees, it sends 
$(\zeta_z)^* M_{y,\sigma,0,\rho}$ to the distinguished irreducible summand $M_{y,z\sigma,0,\rho}$ 
of $E_{y,z\sigma,0,\rho}$.
\end{proof} 

We recall from \cite[Lemma 2.1]{AMS2} that $W_{q\cE} = W_{q\cE}^\circ \rtimes \Gamma_{q\cE}$,
where $W_{q\cE}^\circ$ is the Weyl group of a root system and $\Gamma_{q\cE}$ is the stabilizer
in $W_{q\cE}$ of the set of positive roots. We extend the sign character of $W_{q\cE}^\circ$ to
$W_{q\cE}$ by making it trivial on $\Gamma_{q\cE}$. \\

\textbf{Remark.}
There is an alternative way to extend the sign character of $W_{q\cE}^\circ$ to
$W_{q\cE}$, namely as $\det_{X_* (T)}$. This determinant is a $\Z$-valued character of
$W_{q\cE}$, so it is quadratic. For the purposes of this paper, the above sign character and
$\det_{X_* (T)}$ are equally good. However, the results from \cite[\S 6.2]{SolQS} (which was
written later than this paper) indicate that $\det_{X_* (T)}$ is more natural. \\

To improve the temperedness properties of standard 
$\mh H (G,M,q\cE)$-modules, one can use the Iwahori--Matsumoto involution, given by 
\[
\IM (N_w) = \sgn (w) N_w ,\quad \IM (\mb r) = \mb r ,\quad 
\IM (\xi) = -\xi \qquad w \in W_{q\cE}, \xi \in \mf t^\vee .
\]
Composing a $\mh H$-module with IM changes its $\mc O (\mf t)$-weights by a factor -1.
To compensate for that, in \cite{AMS2,AMS3} the 
authors associate to $(y,\sigma,\rho,r)$ and $(y,\sigma_0,\rho,r)$ the modules
\[
\IM^* E_{y,\textup{d}\gamma_y \matje{r}{0}{0}{-r} - \sigma_0,r,\rho} \qquad \text{and} \qquad
\IM^* M_{y,\textup{d}\gamma_y \matje{r}{0}{0}{-r} - \sigma_0,r,\rho} .
\]
We note that $(\textup{d}\gamma_y \matje{r}{0}{0}{-r} - \sigma_0,r) \in Z_{\mf g \oplus \C}(y)$ and 
\[
(\textup{d}\gamma_y \matje{-r}{0}{0}{r} - \sigma_0,-r) = (-\sigma,-r)  \in Z_{\mf g \oplus \C}(y).
\] 
We define the sign involution of $\mh H$ by 
\[
\sgn (N_w) = \sgn (w) N_w ,\quad \sgn (\mb r) = -\mb r ,\quad 
\sgn |_{\mc O (\mf t)} = \mr{id}_{\mc O (\mf t)} .
\]
The Iwahori--Matsumoto and sign involutions commute and
\[
\sgn \circ \IM = \IM \circ \sgn = \zeta_{-1} : \; \mh H \to \mh H.
\]
From Proposition \ref{prop:5.1} we know that composing with the involution $\sgn \circ \IM =
\zeta_{-1}$ has an easy effect on standard modules. That enables us to reformulate the results of 
\cite{AMS2} which use the Iwahori--Matsumoto involution in terms of the sign involution of $\mh H$.
In particular we see that the module $\sgn^* E_{y,\sigma,-r,\rho}$ is isomorphic to 
$\IM^* E_{y,-\sigma,r,\rho}$. In \cite{AMS2} the latter module was associated to the data 
\[
(y,\sigma_0,r,\rho) \quad \text{and} \quad (y,\sigma_0 + \textup{d}\gamma_y \matje{r}{0}{0}{-r},r,\rho).
\] 
Similar statements holds without $\rho$ and with $E$ replaced by $M$. This is the class of modules 
that is standard in the analytic sense related to the Langlands classification, see 
\cite[\S 3.5]{SolHecke}. To distinguish them from the earlier geometric standard modules, we refer to 
$\sgn^* E_{y,\sigma,-r,\rho}$ as an analytic standard $\mh H (G,M,q\cE)$-module.

Using the sign automorphism we can vary on \cite[Theorem 4.6]{AMS2}. Fix $r \in \C$ and consider 
triples $(y,\sigma,\rho)$ such that:
\begin{itemize}
\item $y \in \mf g$ is nilpotent,
\item $\sigma \in \mf g$ is semisimple and $[\sigma,y] = -2ry$,
\item $\rho \in \Irr \big(\pi_0 (Z_G (\sigma,y)) \big)$ and $q\Psi_{Z_G (\sigma_0)}(y,\rho) =
(M,\cC_v^M,q\cE)$ up to $G$-conjugacy.
\end{itemize}
By Theorem \ref{thm:4.2}.c the map
\begin{equation}\label{eq:5.9}
(\sigma,y,\rho) \mapsto \sgn^* (M_{y,\sigma,-r,\rho})
\end{equation}
defines a bijection from the set of $G$-conjugacy classes of triples $(y,\sigma,\rho)$ as above
to $\Irr_r (\mh H)$. Notice that the central character of $\sgn^* (M_{y,\sigma,-r,\rho})$
is $(\sigma + r \sigma_v, r)$ when $\sigma \in \mf t_{-r}$.
This constitutes an improvement on \cite[\S 3.5]{AMS2} because our new pa\-ra\-me\-trization 
of $\Irr_r (\mh H)$ has all the desired properties with respect to temperedness (see below)
and is more natural -- we do not have to involve d$\gamma_y \matje{r}{0}{0}{-r}$ any more.

\begin{thm}\label{thm:4.3}
\textup{\cite[Theorem 3.25, Theorem 3.26 and \S 4]{AMS2}} \\
Consider an analytic standard $\mh H$-module $\sgn^* E_{y,\sigma,-r,\rho}$.
\enuma{
\item Suppose that $\Re (r) \geq 0$. The $\mh H$-modules $\sgn^* (E_{y,\sigma,-r,\rho})$ and 
$\sgn^* (M_{y,\sigma,-r,\rho})$ are tempered if and only if $\sigma_0$ lies in 
$i \mf t_\R = i \R \otimes_\Z X_* (T)$. 

Here $\sigma_0 = \sigma + \textup{d}\gamma_y \matje{r}{0}{0}{-r}$ 
is as in \eqref{eq:4.4}, but with $-r$ instead of $r$.
\item Suppose that $\Re (r) > 0$. Then $\sgn^* (E_{y,\sigma,-r,\rho})$ and $\sgn^* (M_{y,\sigma,-r,\rho})$ 
are essentially discrete series if and only if $y$ is distinguished in $\mf g$. 

Moreover, when these conditions are fulfilled
\[
\sgn^* (E_{y,\sigma,-r,\rho}) = \sgn^* (M_{y,\sigma,-r,\rho}) \in \Irr_r (\mh H).
\]
}
\end{thm} 

In terms of Theorems \ref{thm:4.2} and \ref{thm:4.3}, 
the bijection from \cite[Theorem 1.5]{SolSGHA} becomes
\begin{equation}
\begin{array}{ccc}
\Irr_r ( \mh H ) & \longrightarrow & \Irr_0 ( \mh H ) \\ 
\sgn^* (M_{y,\sigma,-r,\rho}) & \mapsto & \sgn^* (M_{y,\sigma,0,\rho}) 
\end{array} .
\end{equation}
We would like to analyse the right $\mh H$-modules from Theorem \ref{thm:2.5} 
as left modules over the opposite algebra $\mh H^{op}$. This opposite algebra is easily 
identified via the isomorphism
\begin{equation}\label{eq:5.1}
\begin{array}{ccc@{\qquad}c}
\mh H (G,M,q\cE)^{op} & \isom & \mh H (G,M,q\cE^\vee) & \\
N_w \xi & \mapsto & \xi (N_w)^{-1} & w \in W_{q\cE} , \xi \in \mc O (\mf t \oplus \C) ,
\end{array}
\end{equation}
see \cite[(14)]{AMS2}. That gives an equivalence of categories
\begin{equation}\label{eq:5.11}
\mh H (G,M,q\cE) -\Mod \cong \Mod - \mh H (G,M,q\cE^\vee) .
\end{equation}
The dual local system $q\cE^\vee$ on $\cC_v^M$ is also cuspidal, so all the
previous results hold just as well for $\mh H (G,M,q\cE^\vee)$. In particular we have a complete
classification of its irreducible and its standard left modules.

\subsection{Construction from $K_{\sigma,r}$ or $K_{N,\sigma,r}$} \
\label{par:construction}

We want to relate the standard modules of $\mh H$ (or its opposite) to Theorem \ref{thm:2.5}. 
The vector spaces $H^* (\{y\},i_y^! K_{\sigma,r})$ and $H^* (\{y\},i_y^* K_{\sigma,r})$ become left \\
$\End^*_{\mc D^b_{Z_G (\sigma) \times \C^\times} (\mf g^{\sigma,r})}(K_{\sigma,r})$-modules via the 
natural algebra homomorphism
\begin{equation}\label{eq:5.12}
\End^*_{\mc D^b_C (\mf g^{\sigma,r})}(K_{\sigma,r}) \to
\big( \End^*_{\mc D^b_{C_y^\circ} (\{y\})} (i_y^{! / *} K_{\sigma,r}) \big)^{\pi_0 (C_y)} \to
\End^*_{\mc D^b (\{y\})} (i_y^{! / *} K_{\sigma,r}) ,
\end{equation}
which specializes $H^*_{C_y^\circ} (\{y\})$ at $(\sigma,r)$, see \cite[\S 10.2]{Lus-Cusp2}. 
Via Theorem \ref{thm:2.5}.b, $H^* (i_y^! K_{\sigma,r})$ and $H^* (i_y^* K_{\sigma,r})$ also become left 
$\mh H$-modules. By \cite[\S 10.4]{Lus-Cusp2}, and as in Proposition \ref{prop:2.6}, they 
carry natural actions of $\pi_0 (C_y)$, which commute with the $\mh H$-actions. 
The modules $H^* (\{y\},i_y^! K_{\sigma,r})$ and $H^* (\{y\},i_y^* K_{\sigma,r})$ are annihilated by
$\mb r - r$, so they descend to $\mh H (G,M,q\cE,r)$-modules. From Theorem \ref{thm:8.3} we see that
the action of $\mh H (G,M,q\cE,r)$ can also be constructed more directly, as in \eqref{eq:5.12}
with $Z_G (\sigma)$ instead of $C = Z_G (\sigma) \times \C^\times$.

Let $K_{\sigma,r}^\vee \in \mc D^b_C (\mf g^{\sigma,r})$ be the analogue of $K_{\sigma,r}$, 
but constructed from $q\cE^\vee$. 

\begin{prop}\label{prop:6.2}
Assume Condition \ref{cond:2.1} and denote the subvariety of $\exp(\sigma)$-fixed 
points in $\mc P_y$ by $\mc P_y^\sigma$. 
\enuma{
\item There are natural isomorphisms of $\mh H (G,M,q\cE) \times \pi_0 (C_y)$-representations  
\[
\begin{array}{ccl}
H^* (\{y\}, i_y^! K_{\sigma,r}) & \cong & H_* (\mc P_y^\sigma, \dot{q\cE}) \; \cong \; E_{y,\sigma,r} , \\
H^* (\{y\}, i_y^* K_{\sigma,r}) & \cong & \C_{\sigma,r} \underset{H^*_{C_y^\circ} (\{y\})}{\otimes} 
H^*_{C_y^\circ} (\mc P_y^\sigma, \dot{q\cE}) .
\end{array}
\]
\item There are natural isomorphisms of $\mh H (G,M,q\cE^\vee) \times \pi_0 (C_y)$-representations  
\[
\begin{array}{cclll}
H^* (\{y\}, i_y^! K_{\sigma,r})^\vee & \cong &   
H^* (\{y\}, i_y^* K_{\sigma,r}^\vee ) & \cong & E_{y,\sigma,r}^\vee ,\\
H^* (\{y\}, i_y^* K_{\sigma,r})^\vee & \cong & H^* 
(\{y\}, i_y^! K_{\sigma,r}^\vee) & \cong & E_{y,\sigma,r} ,
\end{array}
\]
not necessarily preserving the gradings.
\item Parts (a) and (b) are also valid with $(\mf g_N^{\sigma,r}, K_{N,\sigma,r})$ instead of
$(\mf g^{\sigma,r}, K_{\sigma,r})$.
}
\end{prop}
\begin{proof}
(a) When $G$ is connected, the isomorphisms with $H^* (\{y\}, i_y^! K_{\sigma,r})$ are shown in 
\cite[Proposition 10.12]{Lus-Cusp2}. We generalize those arguments to our setting. 
Consider the pullback diagram \vspace{-2mm}
\[
\xymatrix{
\mc P_y^\sigma \ar[r]^{k} \ar[d]^{\pi} & \dot{\mf g}^{\sigma,r} \ar[d]^{\pr_1} \\
\{y\} \ar[r]^{i_y} & \mf g^{\sigma,r} 
}
\]
where $k (gP) = (y,gP)$. By the conventions at the start of Section \ref{sec:standard} we have
$k^* \dot{q\cE} = \dot{q\cE}$, where the first $\dot{q\cE}$ lives on $\dot{\mf g}_{\sigma,r}$ 
and the second on $\mc P_y^\sigma$.
From general results about derived sheaves, \cite[\S 1.4.6 and Theorem 1.8.ii]{BeLu}, extended to 
the equivariant derived category in \cite[Theorem 3.4.3]{BeLu}, it is known that
\begin{equation}\label{eq:5.10}
i_y^* \pr_{1,*} = \pi_* k^* \qquad \text{and} \qquad i_y^! \pr_{1,*} = \pi_* k^!
\end{equation}
as functors $\mc D^b_C (  \dot{\mf g}^{\sigma,r} ) \to \mc D^b_{C_y} (\{y\})$. With that we compute
\[
\begin{array}{lllllll}
H^* (\{y\}, i_y^! K_{\sigma,r}) & \cong & H^* (\{y\}, i_y^! \pr_{1,*} \dot{q\cE}) & \cong &
H^* (\{y\}, \pi_* k^! \dot{q\cE}) \\
& \cong & H^* (\mc P_y^\sigma, k^! \dot{q\cE}) & \cong & H^* (\mc P_y^\sigma, D k^* D \dot{q\cE}) \\
& \cong & H^* (\mc P_y^\sigma, D k^* \dot{q\cE}^\vee ) & = & H^* (\mc P_y^\sigma, D \dot{q\cE}^\vee)
& = & H_* (\mc P_y^\sigma, \dot{q\cE}) .
\end{array}
\]
The last part of the proof of \cite[Proposition 10.12]{Lus-Cusp2} shows that
\[
H_* (\mc P_y^\sigma, \dot{q\cE}) \cong \C_{\sigma,r} \underset{H^*_{C_y^\circ} (\{y\})}{\otimes} 
H_*^{C_y^\circ} (\mc P_y^\sigma, \dot{q\cE}) \cong \C_{\sigma,r} \underset{H^*_{C_y^\circ} (\{y\})}{
\otimes} H_*^{C_y^\circ} (\mc P_y, \dot{q\cE}) = E_{y,\sigma,r}
\]
as $\mh H (G,M,q\cE^\vee) \times \pi_0 (C_y)$-representations.
Similarly we use \eqref{eq:5.10} to compute
\[
\begin{array}{lllll}
H^* (\{y\}, i_y^* K_{\sigma,r}) & \cong & H^* (\{y\}, i_y^* \pr_{1,*} \dot{q\cE}) & \cong &
H^* (\{y\}, \pi_* k^* \dot{q\cE}) \\
& \cong & H^* (\{y\}, \pi_* \dot{q\cE}) & \cong & H^* (\mc P_y^\sigma, \dot{q\cE}) .
\end{array}
\]
Notice that $\mc P_y^\sigma$ is compact, so cohomology coincides with compactly supported cohomology 
here. The last part of the proof of \cite[Proposition 10.12]{Lus-Cusp2} also shows that there is an
isomorphism of $\mh H (G,M,q\cE^\vee) \times \pi_0 (C_y)$-representations
\[
H^* (\mc P_y^\sigma, \dot{q\cE}) \cong \C_{\sigma,r} \underset{H^*_{C_y^\circ} (\{y\})}{\otimes} 
H^*_{C_y^\circ} (\mc P_y^\sigma, \dot{q\cE}) \cong
\C_{\sigma,r} \underset{H^*_{C_y^\circ} (\{y\})}{\otimes} H^*_{C_y^\circ} (\mc P_y, \dot{q\cE}) .
\]
(b) By \eqref{eq:2.21}, \eqref{eq:2.13} and the properties of Verdier duality \cite[\S 2.8 and Lemma
3.3.13]{Ach} there are natural isomorphisms (which may shift the gradings by different amounts
on different simple summands)
\begin{equation}\label{eq:6.1}
\begin{aligned}
D K_{\sigma,r} & = \pr_{1,*} D \IC_{G \times \C^\times} (\mf g^{\sigma,r} \times 
(G/P)^{\exp (\C \sigma)}, \dot{q\cE}) \\
& \cong \pr_{1,!} \IC_{G \times \C^\times} (\mf g^{\sigma,r} \times 
(G/P)^{\exp (\C \sigma)}, D \dot{q\cE}) \\
& \cong \pr_{1,!} \IC_{G \times \C^\times} (\mf g^{\sigma,r} \times 
(G/P)^{\exp (\C \sigma)}, \dot{q\cE^\vee}) \; = \; K^\vee_{\sigma,r}.
\end{aligned}
\end{equation}
From part (a) and \eqref{eq:6.1} we see that 
\begin{align*}
H^* (\{y\}, i_y^! K_{\sigma,r})^\vee & \cong H^{-*} (\{y\}, D i_y^! K_{\sigma,r}) \cong
H^{-*} (\{y\}, i_y^* D K_{\sigma,r}) \\
& \cong H^{-*} (\{y\}, i_y^* K_{\sigma,r}^\vee ) \cong H^{-*} (\{y\}, i_y^* K_{\sigma,r}^\vee) .
\end{align*}
Here $-*$ means that initially the grading is reversed (by $\vee$), while in the last line 
the grading must also be adjusted to account for \eqref{eq:6.1}. 
Similarly there is a natural vector space isomorphism
\[
H^* (\{y\}, i_y^* K_{\sigma,r})^\vee \cong H^{-*} (\{y\}, i_y^! K_{\sigma,r}^\vee) .
\]
The $\mh H\times \pi_0 (C_y)$-actions in 
part (a) become actions of $\mh H^{op} \cong \mh H (G,M,q\cE^\vee )$ and of $\pi_0 (C_y)$ 
upon taking vector space duals. We can reformulate the isomorphisms from part (a) as
\begin{equation}\label{eq:6.2}
\begin{array}{ccc}
H_* (\mc P_y^\sigma, \dot{q\cE})^\vee & \cong & H^{- *}(\mc P_y^\sigma, \dot{q\cE^\vee}) ,\\
H^* (\mc P_y^\sigma, \dot{q\cE})^\vee & \cong & H_{- *}(\mc P_y^\sigma, \dot{q\cE^\vee}) .
\end{array} 
\end{equation}
Using the explicit description of the actions given in \cite[\S 2.1]{SolSGHA} and in \cite{AMS2},
one checks readily that in \eqref{eq:6.2} we have isomorphisms of  
$\mh H (G,M,q\cE^\vee) \times \pi_0 (C_y)$-representations. \\
(c) This can be shown in the same way as parts (a) and (b).
\end{proof}

With Proposition \ref{prop:6.2} we can henceforth interpret the geometric standard 
$\mh H$-module $E_{y,\sigma,r,\rho}$ as 
\begin{equation}\label{eq:6.12}
\Hom_{\pi_0 (C_y)} \big( \rho, H^* (\{y\},i_y^! K_{N,\sigma,r}) \big) \cong 
\C_{\sigma,r} \underset{H_{C_y^\circ}^* (\{y\})}{\otimes} \Hom^*_{\mc D^b_{C_y} (\{y\})} 
\big( \rho, i_y^! K_{N,\sigma,r} \big) .
\end{equation}
With Theorem \ref{thm:8.3} we can reformulate \eqref{eq:6.12} as an isomorphism of 
$\mh H (G,M,q\cE,r)$-modules:
\[
E_{y,\sigma,r,\rho} \cong \C_\sigma \underset{H_{Z_G^\circ (\sigma)}^* (\{y\})}{\otimes} 
\Hom^*_{\mc D^b_{Z_G (\sigma,y)} (\{y\})} \big( \rho, i_y^! K_{N,\sigma,r} \big) .
\]

\section{Structure of the localized complexes $K_{\sigma,r}$ and $K_{N,\sigma,r}$}
\label{sec:structure}

With $K_{\sigma,r}$ and $K_{\sigma,r}^\vee$ we can give an alternative interpretation of the cuspidal 
quasi-supports involved in standard modules (see in particular Theorem \ref{thm:4.2}.c). 

\begin{prop}\label{prop:2.6}
Fix a nilpotent $y \in \mf g^{\sigma,r}$ and let $i_y : \{y\} \to \mf g^{\sigma,r}$ be the 
inclusion. For $\rho \in \Irr \big( \pi_0 (Z_G (y,\sigma_0)) \big) = \Irr ( \pi_0 (C_y) )$, 
the following are equivalent:
\begin{enumerate}[(i)]
\item the cuspidal quasi-support $q\Psi_{Z_G (\sigma_0)}(y,\rho)$, with respect to the
group $Z_G (\sigma_0)$, is $G$-conjugate to $(M,\cC_v^M,q\cE)$,
\item $\Hom_{\pi_0 (C_y)} \big( \rho, H^* (\{y\}, i_y^! K_{\sigma,r}) \big) \neq 0$,
\item $\Hom_{\pi_0 (C_y)} \big( H^* (\{y\},i_y^* K_{\sigma,r}), \rho \big) \neq 0$,
\item $\Hom_{\pi_0 (C_y)} \big( \rho, H^* (\{y\}, i_y^! K_{N,\sigma,r}) \big) \neq 0$,
\item $\Hom_{\pi_0 (C_y)} \big( H^* (\{y\},i_y^* K_{N,\sigma,r}), \rho \big) \neq 0$.
\end{enumerate}
\end{prop}
\begin{proof}
Recall from Proposition \ref{prop:6.2}.a that 
\[
H^* (\{y\}, i_y^! K_{\sigma,r}) \cong E_{y,\sigma,r} .
\]
With that in mind, the equivalence of (i) and (ii) is shown in \cite[Proposition 3.7]{AMS2} when 
$G$ is connected. With \cite[\S 4]{AMS2} those arguments can be extended to disconnected $G$ and 
cuspidal \emph{quasi}-supports. The equivalence of (ii) and (iii) follows
from Proposition \ref{prop:6.2}.b. By Proposition \ref{prop:6.2}.(a,c)
\[
H^* (\{y\}, i_y^! K_{\sigma,r}) \cong H^* (\{y\}, i_y^! K_{N,\sigma,r}) \quad \text{and} \quad
H^* (\{y\}, i_y^* K_{\sigma,r}) \cong H^* (\{y\}, i_y^* K_{N,\sigma,r})  
\]
as $\mh H (G,M,q\cE) \times \pi_0 (C_y)$-representations. That proves the equivalence of (ii) with
(iv) and of (iii) with (v).
\end{proof}

Cuspidal quasi-supports were defined \cite[\S 5]{AMS1}, in relation with a Springer correspondence
for disconnected reductive groups. In our context, it is more convenient to use the property (ii)
or (iii) in Proposition \ref{prop:2.6}: for a given triple $(y,\sigma,\rho)$ that determines
$(M,\cC_v^M, q\cE)$ up to $G$-conjugacy. 

In the opposite direction, Proposition \ref{prop:2.6} almost determines the semisimple complex 
$K_{\sigma,r}$. To work this out, let $\mc O_y = \Ad (C) y \subset \mf g^{\sigma,r}$ be the 
$C$-orbit of $y$. Regarding $\rho$ as a $C_y$-equivariant sheaf on $\{y\}$ and invoking the 
equivalence of categories
\begin{equation}\label{eq:6.3}
\mr{ind}_{C_y}^C : \mc D^b_{C_y} (\{y\}) \isom \mc D^b_C (\mc O_y) ,
\end{equation}
we obtain a $C$-equivariant local system $\mr{ind}_{C_y}^C (\rho)$ on $\mc O_y$. We form the 
equivariant intersection cohomology complex $\IC_C \big( \mf g^{\sigma,r}, \ind_{C_y}^C (\rho) \big) 
\in \mc D^b_C (\mf g^{\sigma,r})$, which is supported on $\overline{\mc O_y}$. This is the usual 
intersection cohomology complex $\IC \big( \mf g^{\sigma,r}, \mr{ind}_{C_y}^C (\rho) \big)$, 
only now considered with its $C$-equivariant structure.

\begin{thm}\label{thm:2.7}
\enuma{
\item Fix $r \in \C^\times$. Every simple direct summand of $K_{\sigma,r}$ is isomorphic to 
$\IC_C \big( \mf g^{\sigma,r}, \mr{ind}_{C_y}^C (\rho) \big)$, for data $(y,\sigma,\rho)$ 
that fulfill the conditions in Proposition \ref{prop:2.6}. Conversely, every such equivariant 
intersection cohomology complex is a direct summand of $K_{\sigma,r}$ (with multiplicity $\geq 1$).
\item For arbitrary $r \in \C$, part (a) becomes valid when we replace all involved sheaves by 
their versions for $\mf g_N^{\sigma,r}$.
}
\end{thm}
\begin{proof}
(a) From \cite[(95)]{AMS2} we see that, as $M^\circ$-equivariant local system on 
$\cC_v^M = \cC_v^{M^\circ}$, $q\cE$ is a direct sum of $M$-conjugates of $\mc E$ and 
$(q\cE)_v \cong \mc E_v \rtimes \rho_M$ for a suitable representation $\rho_M$. 
Hence $q\cE \in \mc D^b_{G^\circ}(\dot{\mf g})$ is a direct sum of $G$-conjugates of $\dot{\cE} \in 
\mc D^b_{G^\circ} (\dot{\mf g}^\circ)$. 

Then the diagram \eqref{eq:2.12} shows that, as an element of $\mc D^b_{Z_G^\circ (\sigma) \times 
\C^\times} (\mf g^{\sigma,r})$, $K_{\sigma,r}$ is a direct sum of $Z_G (\sigma)$-conjugates of
$K_{\sigma,r}^\circ$ -- the version of $K_{\sigma,r}$ for $(G^\circ, M^\circ, \cE)$. 
By \cite[\S 5.3]{Lus-Cusp2}, $K_{\sigma,r}^\circ$ is a semisimple complex of sheaves. Further 
\cite[Proposition 8.17]{Lus-Cusp2} (for which we need $r \neq 0$) and Proposition \ref{prop:2.6} 
entail that the simple direct summands of $K_{\sigma,r}^\circ$ are the 
$Z_G^\circ (\sigma) \times \C^\times$-equivariant intersection cohomology complexes 
\begin{equation}\label{eq:2.15}
\IC_{C^\circ}\big( \mf g^{\sigma,r}, \mr{ind}^{C^\circ}_{Z_{G^\circ} (\sigma,y)} (\rho^\circ) \big) 
\quad \text{with} \quad
\Hom_{\pi_0 (C_y)}(H^* (\{y\},i_y^* K_{\sigma,r}^\circ), \rho^\circ ) \neq 0 .
\end{equation}
More precisely, every such summand appears with a multiplicity $\geq 1$. Then $K_{\sigma,r}$ is a direct 
sum of terms 
\[
\IC_{C^\circ}\big( \mf g^{\sigma,r}, \Ad (g)^* \mr{ind}^{C^\circ}_{Z_{G^\circ} (\sigma,y)} (\rho^\circ) \big),
\] 
where $g \in Z_G (\sigma)$ and $(y,\rho)$ are as in \eqref{eq:2.15}. Again every such summand appears with
multiplicity $\geq 1$ in $K_{\sigma,r}$. 

On the other hand, we already knew that $K_{\sigma,r}$ is a $C$-equivariant semisimple complex of sheaves.
We deduce that $K_{\sigma,r}$ is a direct sum of terms 
$\IC_C \big( \mf g^{\sigma,r}, \mr{ind}_{C_y}^C (\rho' ) \big)$, where 
$\rho' \in \Irr (\pi_0 (C_y))$ contains some $\rho^\circ$ as before. That settles the 
geometric structure of $K_{\sigma,r}$, it remains to identify exactly which $\rho'$ occur.

The above works equally well with the group $G^\circ M$ instead of $G$. Let us assume that 
$\sigma_0, \sigma -r \sigma_v \in \mf t$, as we may by \cite[Proposition 1.4.c]{AMS3}. Then 
\cite[Lemma 4.4]{AMS2} says that every $\rho^\circ$ as in \eqref{eq:2.15} corresponds to a unique 
\[
\rho^\circ \rtimes \rho_M \in \Irr \big( \pi_0 (Z_{G^\circ M}(\sigma_0,y)) \big)
\] 
with $q\Psi_{Z_G^\circ (\sigma_0) M} (y, \rho^\circ \rtimes \rho_M)$ conjugate to $(M,\cC_v^M,q\cE)$ -- 
see also \eqref{eq:A.12}.
A direct comparison of the constructions of $K_{\sigma,r}^\circ$ and of $K_{\sigma,r}$ for $G^\circ M$
shows that the latter equals the direct sum of the complexes
\[
\IC_{Z_{G^\circ M}(\sigma) \times \C^\times} \big( \mf g^{\sigma,r}, 
\mr{ind}_{C_y}^C  (\rho^\circ \rtimes \rho_M) \big) ,
\]
with the same multiplicities as for $K_{\sigma,r}^\circ$. 

The step from $K_{\sigma,r}$ for $G^\circ M$ to $K_{\sigma,r}$ for $G$ is just induction, compare with
\eqref{eq:A.11}. This induction preserves the cuspidal quasi-supports (for $G$) from \cite[\S 5]{AMS1}, 
because those are based on what happens for objects coming from $G^\circ M$ (when this support comes
from $M$). We conclude that $K_{\sigma,r}$ (for $G$) is a direct sum of terms 
\[
\IC_C \big( \mf g^{\sigma,r}, 
\mr{ind}_{C_y}^C  \big( \mr{ind}^{C_y}_{C_y \cap G^\circ M} (\rho^\circ \rtimes \rho_M) \big) \big) ,
\]
with multiplicities coming from \eqref{eq:2.15}. 
In particular $K_{\sigma,r}$ is also a direct sum of (degree shifts of) simple perverse sheaves 
\begin{equation}\label{eq:2.16}
\IC_C \big( \mf g^{\sigma,r}, \mr{ind}_{C_y}^C  (\rho ) \big) \quad \text{where} \quad 
q\Psi_{Z_G (\sigma_0)}(y,\rho) = [M,\cC_v^M,q\cE ]_G. 
\end{equation}
By Frobenius reciprocity, applied to $\mr{ind}^{C_y}_{C_y \cap G^\circ M} (\rho^\circ \rtimes \rho_M)$, 
every term \eqref{eq:2.16} appears with a multiplicity $\geq 1$ in $K_{\sigma,r}$.\\
(b) This can be shown in the same way, if we the replace the crucial input from 
\cite[Proposition 8.17]{Lus-Cusp2} by \cite[\S 9.5]{Lus-Cusp2}.
\end{proof}

We note that by Theorem \ref{thm:2.7}.b, every simple perverse sheaf in $\mc D^b_{Z_G (\sigma) \times
\C^\times}(\mf g_N^{\sigma,r})$ occurs (maybe with a degree shift) as a summand of $K_{N,\sigma,r}$,
for a suitable cuspidal support $(M,\cC_v^M,q\cE)$. With the complexes $K_{\sigma,r}$ or 
$K_{N,\sigma,r}$, we can construct standard modules in yet another way.

\begin{lem}\label{lem:6.3} 
Assume that $(y,\rho)$ fulfills the equivalent conditions in Proposition \ref{prop:2.6} and let 
$j : \mc O_y \to \mf g^{\sigma,r}$ be the inclusion.
\enuma{
\item There are natural isomorphisms of $\mh H (G,M,q\cE^\vee)$-modules
\begin{align*}
& \Hom^*_{\mc D^b_C (\mf g^{\sigma,r})} \big( K_{\sigma,r}, j_* \mr{ind}_{C_y}^C  (\rho) \big) 
\cong \big( H^*_{C_y^\circ}(\{y\}) \otimes_\C H^* ( (i_y^* K_{\sigma,r})^\vee \otimes 
\rho ) \big)^{\pi_0 (C_y)} \\
& \C_{\sigma,r} \underset{H_C^* (\pt)}{\otimes} \Hom^*_{\mc D^b_C (\mf g^{\sigma,r})} 
\big( K_{\sigma,r}, j_* \mr{ind}_{C_y}^C  (\rho) \big) \cong E_{y,\sigma,r,\rho^\vee} .
\end{align*}
The former is an isomorphism of graded modules.
\item The isomorphisms from part (a) are also valid for $K_{N,\sigma,r}$ and the inclusion  
$j_N : \mc O_y \to \mf g_N^{\sigma,r}$.
}
\end{lem}
\begin{proof}
(a) By adjunction and \eqref{eq:6.3} there are natural isomorphisms 
\begin{equation}\label{eq:6.6}
\begin{aligned}
\Hom^*_{\mc D^b_C (\mf g^{\sigma,r})} \big( K_{\sigma,r}, j_* \mr{ind}_{C_y}^C  (\rho) \big) & 
\cong \Hom^*_{\mc D^b_C (\mc O_y)} \big( j^* K_{\sigma,r}, \mr{ind}_{C_y}^C  (\rho) \big) \\
& \cong \Hom^*_{\mc D^b_{C_y} (\{y\})} (i_y^* K_{\sigma,r}, \rho) .
\end{aligned} \vspace{-2mm}
\end{equation}
With \cite[\S 1.10]{Lus-Cusp2} it can be rewritten as 
\begin{equation}\label{eq:6.4}
\begin{aligned}
\big( \Hom^*_{\mc D^b_{C_y^\circ} (\{y\})} (i_y^* K_{\sigma,r}, \rho) \big)^{\pi_0 (C_y)} 
\; \cong \; 
& \big( H^*_{C_y^\circ} (\{y\}, D i_y^* K_{\sigma,r} \otimes_\C \rho ) \big)^{\pi_0 (C_y)} \\
\; \cong \; & 
\big( H^*_{C_y^\circ} (\{y\}, D i_y^* K_{\sigma,r} ) \otimes_\C \rho \big)^{\pi_0 (C_y)} .
\end{aligned}
\end{equation}
By \cite[\S 1.21]{Lus-Cusp2}, \eqref{eq:6.4} is isomorphic with
\begin{equation}\label{eq:6.7}
\big( H^*_{C_y^\circ}(\{y\}) \otimes_\C H^* (\{y\}, 
D i_y^* K_{\sigma,r}) \otimes \rho \big)^{\pi_0 (C_y)} ,
\end{equation}
which gives the first isomorphism of the statement.

Since $C_y$ acts trivially on $\C_{\sigma,r}$, we can tensor the isomorphisms 
\eqref{eq:6.6}--\eqref{eq:6.7} with $\C_{\sigma,r}$ over $H^*_{C_y^\circ}(\{y\})$. That preserves 
the structure as left $\mh H (G,M,q\cE^\vee)$-module or right $\mh H (G,M,q\cE)$-module, but it
destroys the grading unless $(\sigma,r) = (0,0)$. It does not matter whether we tensor with 
$\C_{\sigma,r}$ before or after taking $\pi_0 (C_y)$-invariants. Thus it transforms \eqref{eq:6.7} into
\[
\big( H^* (\{y\}, D i_y^* K_{\sigma,r} ) \otimes \rho \big)^{\pi_0 (C_y)} \cong
\big( H^{-*} (\{y\}, i_y^* K_{\sigma,r} )^\vee \otimes \rho \big)^{\pi_0 (C_y)} .
\]
By Proposition \ref{prop:6.2} that is isomorphic with
\[
\big( H^{*} (\{y\}, i_y^! K_{\sigma,r}^\vee ) \otimes \rho 
\big)^{\pi_0 (C_y)} \cong ( E_{y,\sigma,r} \otimes \rho )^{\pi_0 (C_y)} .
\]
This can also be interpreted as 
$\Hom_{\pi_0 (C_y)}(\rho^\vee , E_{y,\sigma,r}) = E_{y,\sigma,r,\rho^\vee}$.\\
(b) The same argument as for part (a) works.
\end{proof}

\begin{rem}\label{rem:6.4}
For any $y \in \mf g_N^{\sigma,r}$
\begin{equation}\label{eq:8.12}
\C^\times y \text{ is contained in } \Ad (Z_G^\circ (\sigma) ) y ,
\end{equation}
because $y$ is part of a $\mf{sl}_2$-triple in $Z_{\mf g}(\sigma)$. In particular $Z_G (\sigma)$
and $C = Z_G (\sigma) \times \C^\times$ have the same orbits on $\mf g_N^{\sigma,r}$. Recall from 
\cite[Lemma 3.6.a]{AMS2} that $\pi_0 (C_y) \cong \pi_0 (Z_G (\sigma,y))$. 
For these reasons the results in Section \ref{sec:structure} 
remain valid when we replace $C$ by $Z_G (\sigma)$ everywhere. 
\end{rem}

\section{The Kazhdan--Lusztig conjecture}
\label{sec:KL}

The properties of $K_{N,\sigma,r}$ can be used to compute multiplicities between irreducible and 
standard modules. That enables us to investigate the Kazhdan--Lusztig conjecture \cite[\S 8]{Vog} for
graded Hecke algebras. For $\pi \in \Irr (\mh H (G,M,q\cE))$, write 
\[
\mu (\pi, E_{y,\sigma,r,\rho}) = \text{ multiplicity of } \pi \text{ in } E_{y,\sigma,r,\rho},
\] 
computed in the Grothendieck group of $\Mod_{\mr{fl}} (\mh H (G,M,q\cE))$. When $\mb r - r$ annihilates 
$\pi$, we can of course compute $\mu (\pi, E_{y,\sigma,r,\rho})$ just as well in the Grothendieck
group of $\Mod_{\mr{fl}} (\mh H (G,M,q\cE,r))$. In relation with the
analytic standard modules from Theorem \ref{thm:4.3} we record the obvious equality
\begin{equation}\label{eq:2.20}
\mu (\sgn^* \pi, \sgn^* E_{y,\sigma,r,\rho}) = \mu (\pi, E_{y,\sigma,r,\rho}) .
\end{equation}
Let $y' \in \mf g_N^{\sigma,r}$ and $\rho' \in \Irr \, \pi_0 (Z_G (\sigma,y'))$. Then 
$\mr{ind}_{C_{y'}}^C (\rho')$ 
is an irreducible $C$-equivariant local system on $\mc O_{y'} = \Ad (C) y'$. We define
\[
\mu \big( \mr{ind}_{C_y}^C  (\rho), \mr{ind}_{C_{y'}}^C (\rho') \big) = 
\text{multiplicity of } \mr{ind}_{C_y}^C (\rho) \text{ in } 
\mc H^* \big( \IC_C (\mf g_N^{\sigma,r}, \mr{ind}_{C_{y'}}^C (\rho')) \big) \big|_{\mc O_y}  .
\]
The notations on the right hand side mean that we build a $C$-equivariant intersection
cohomology complex from $\rho'$, we take its cohomology sheaves and we pull those back to $\mc O_y$.
With Remark \ref{rem:6.4}, we can also regard $\mr{ind}_{C_{y'}}^C (\rho')$ as the irreducible 
$Z_G (\sigma)$-equivariant local system $\mr{ind}_{Z_G (\sigma,y')}^{Z_G (\sigma)}(\rho')$ on 
$\mc O_{y'} = \Ad (Z_G (\sigma)) y'$. Then we can define \\
$\mu \big( \mr{ind}_{Z_G (\sigma,y)}^{Z_G (\sigma)} (\rho), 
\mr{ind}_{Z_G (\sigma,y')}^{Z_G (\sigma)} (\rho') \big)$ as
\[
\text{the multiplicity of } \mr{ind}_{Z_G (\sigma,y)}^{Z_G (\sigma)} (\rho) \text{ in }
\mc H^* \big( \IC_{Z_G (\sigma)} (\mf g_N^{\sigma,r}, 
\mr{ind}_{Z_G (\sigma,y')}^{Z_G (\sigma)} (\rho')) \big) \big|_{\mc O_y}  .
\]
Replacing $C$ by $Z_G (\sigma)$ does not really change the involved equivariant intersection complexes,
so we conclude that
\begin{equation}\label{eq:8.16}
\mu \big( \mr{ind}_{C_y}^C  (\rho), \mr{ind}_{C_{y'}}^C (\rho') \big) = 
\mu \big( \mr{ind}_{Z_G (\sigma,y)}^{Z_G (\sigma)} (\rho), 
\mr{ind}_{Z_G (\sigma,y)}^{Z_G (\sigma)} (\rho') \big) .
\end{equation}
Proposition \ref{prop:2.6} says that, if $(y',\rho')$ does not fulfill the conditions stated there: 
\[
\mu (\mr{ind}_{C_y}^C  (\rho), \mr{ind}_{C_{y'}}^C (\rho')) = 0. 
\]
That is not surprising, because in that case $(y',\rho')$ does not correspond to any 
$\mh H (G,M,q\cE)$-module.

\begin{prop}\label{prop:2.9}
In the above setup, assume that both $(y,\rho)$ and $(y',\rho')$ satisfy the equivalent conditions in 
Proposition \ref{prop:2.6}. 
\enuma{
\item $\mu (M_{y',\sigma,r,\rho'}, E_{y,\sigma,r,\rho}) = \mu \big( \mr{ind}_{C_y}^C (\rho), 
\mr{ind}_{C_{y'}}^C (\rho') \big)$.
\item The same holds if we replace the standard module $E_{y,\sigma,r,\rho}$ by the ``costandard
module" $\Hom_{\pi_0 (C_y)} \big( \rho, H^* ( \{y\}, i_y^* K_{N,\sigma,r}) \big)$. 
}
\end{prop}
\begin{proof}
In the cases where $G$ and $M$ are connected and $r \neq 0$, this is proven in 
\cite[\S 10.4--10.8]{Lus-Cusp2}. With Proposition \ref{prop:2.6} and Theorem \ref{thm:2.7} available, 
these arguments from \cite{Lus-Cusp2} remain valid in our generality. We remark that, since we work 
with $\mf g_N^{\sigma,r}$ instead $\mf g^{\sigma,r}$, no extra problems arise when $r = 0$. 
\end{proof}

Proposition \ref{prop:2.9} and \eqref{eq:2.20} establish a version of the Kazhdan--Lusztig 
conjecture for (twisted) graded Hecke algebras of the form $\mh H (G,M,q\cE)$ or $\mh H (G,M,q\cE,r)$. 
In view of \eqref{eq:8.16}, we may also interpret the geometric multiplicities as computed with
$Z_G (\sigma)$-equivariant constructible sheaves. That fits well with Paragraph \ref{par:fixed}, 
in particular with Theorem \ref{thm:8.3}.

From here we would like to establish cases of the Kazhdan--Lusztig conjecture for $p$-adic groups
\cite[Conjecture 8.11]{Vog} (but without the sign involved over there, in our context such a sign would 
be superfluous). It remains to look for instances of a local Langlands correspondence which 
run via an algebra of the form $\mh H (G,M,q\cE) / (\mb r - r)$. 

We will now discuss in which cases this is known, and the setup needed to get there. Let $F$ be a
non-archimedean local field and let $\mc G$ be a connected reductive group defined over $F$.
Let $\mc M$ be a $F$-Levi subgroup of $\mc G$ and let $\tau \in \Irr (\mc M (F))$ be supercuspidal.
This already gives rise to the category $\Rep_{\mr{fl}}(\mc G (F))^{\tau}$ of finite length
smooth $\mc G(F)$-representations all whose irreducible subquotients have cuspidal support conjugate 
to $(\mc M (F),\tau)$. Assume now that $\tau$ is tempered, write
\[
X_\nr^+ (\mc M (F)) = \Hom (\mc M (F), \R_{>0})
\]
and let $\Rep_{\mr{fl}}(\mc G (F))^{\tau +}$ be the category of all finite length smooth
$\mc G (F)$-representations whose cuspidal support is contained in the $\mc G (F)$-orbit of
$\big( \mc M (F), \tau X_\nr^+ (\mc M (F)) \big)$. The set of irreducible objects of
$\Rep_{\mr{fl}}(\mc G (F))^\tau$ will be denoted $\Irr (\mc G (F))^\tau$, and likewise with $\tau +$.

To the data $(\mc G (F), \mc M (F), \tau)$ one can associate a twisted graded Hecke algebra 
$\mh H_\tau$, such that there is an equivalence of categories
\begin{equation}\label{eq:8.15}
\Rep_{\mr{fl}}(\mc G (F))^{\tau +} \cong \mh H_\tau - \Modf{\mf a} , 
\end{equation}
see \cite[Corollary 8.1]{SolEnd}. Here $\Modf{\mf a}$ means finite length right modules
with all $\mc O (\mf t)$-weights in $\mf a$, and one may identify
\[
\mf a = \mr{Lie}\big( X_\nr^+ (\mc M (F)) \big) = \Hom (\mc M (F), \R) .  
\]

\begin{thm}\label{thm:8.5}
In the above setting, suppose that $\mh H_\tau^{op}$ is of the form $\mh H (G,M,q\cE,r)$
for some $r \in \R$. Fix $\sigma_0 \in \mf a$ and write $\sigma = \sigma_0 - r \sigma_v$.
\enuma{
\item There is an equivalence of categories
\[
\Rep_{\mr{fl}}(\mc G (F))^{\tau +} \cong \Modf{\mf a} (\mh H (G,M,q\cE,r)) . 
\]
\item There is an equivalence of categories 
\[
\Rep_{\mr{fl}} (\mc G (F))^{\tau \otimes \exp (\sigma_0)} \cong 
\Modf{\sigma} \big( \End^*_{\mc D^b_{Z_G (\sigma)} (\mf g^{\sigma,-r}_N)} (K_{N,\sigma,-r}) \big) .
\]
\item The Kazhdan--Lusztig conjecture holds for 
$\Rep_{\mr{fl}} (\mc G (F))^{\tau \otimes \exp (\sigma_0)}$, in the form 
\[
\mu (\sgn^* M_{y',\sigma,-r,\rho'}, \sgn^* E_{y,\sigma,-r,\rho}) = \mu \big( \mr{ind}_{Z_G 
(\sigma,y)}^{Z_G (\sigma)} (\rho), \mr{ind}_{Z_G (\sigma,y')}^{Z_G (\sigma)} (\rho') \big) 
\]
where the right hand side is computed in $\mc D^b_{Z_G (\sigma)} (\mf g^{\sigma,-r}_N)$.
}
\end{thm}
\begin{proof}
(a) This is an obvious consequence of \eqref{eq:8.15}.\\
(b) Apply $\sgn^*$ and Theorem \ref{thm:8.3}.c to part (a).\\
(c) This follows from \eqref{eq:2.20}, part (b), Proposition \ref{prop:2.9}.a and \eqref{eq:8.16}.
\end{proof}

We note that by \cite{SolEnd,SolParam} the $k$-parameters of the algebras $\mh H_\tau$ are very
often (conjecturally always) of the required kind. The analysis of the 2-cocycles of the 
group $W_{q\cE}$ for $\mh H_\tau$ may be difficult sometimes, but fortunately these 2-cocycles
are trivial in most cases. Therefore the assumption of Theorem \ref{thm:8.5} is fulfilled for
large classes of groups $\mc G$ and representations, and we expect that it holds always.\\

Next we suppose that a local Langlands correspondence is known for sufficiently large classes of
representations of $\mc M (F)$ and of $\mc G (F)$ that is, for some supercuspidal 
$\mc M (F)$-representations and for all the resulting Bernstein components of $\Irr (\mc G (F))$.
Let $\mc G^\vee$ and $\mc M^\vee$ be the complex dual groups of $\mc G$ and $\mc M$. 
Let $(\phi,\rho)$ be the enhanced 
L-parameter of $\tau$, so $\phi$ takes values in $\mc M^\vee \rtimes \mb W_F$. The group 
$X_\nr^+ (\mc M (F))$ embeds naturally in $Z(\mc M^\vee)$, and the latter acts on the set of
Langlands parameters for $\mc M (F)$ by adjusting the image of a Frobenius element.

To $(\mc G^\vee \rtimes \mb W_F, \mc M^\vee \rtimes \mb W_F, \phi,\rho)$ one can associate a 
triple $(G,M,q\cE)$ as throughout this paper \cite[\S 3.1]{AMS3}, and a twisted graded Hecke algebra 
$\mh H_{\phi,\rho}$ of the form $\mh H (G,M,q\cE)$. The involved group $G$ is called 
$G_{\phi_b} \times X_\nr ( {}^L \mc G )$ in \cite[(3.5)]{AMS3}, it is a finite cover of 
$Z_{\mc G^\vee}(\phi (\mb W_F))$. An important property of this algebra is:

\begin{thm}\label{thm:8.6} 
Fix $r \in \R$.
\enuma{
\item There exists a canonical bijection between $\Irr_{\mf a}(\mh H_{\phi,\rho} / (\mb r - r))$ 
and the set of enhanced L-parameters for $\mc G (F)$ (or an inner twist thereof) whose cuspidal 
support is $\mc G^\vee$-conjugate to an element of $(\mc M^\vee, X_\nr^+ (\mc M (F)) \phi,\rho)$. 
\item The sets in part (a) are canonically in bijection with 
$\Irr_{\mf a}(\mh H_{\phi,\rho^\vee} / (\mb r - r))$, via taking contragredients of enhancements
of L-parameters.
}
\end{thm}
\begin{proof}
(a) This is \cite[Theorem 3.8]{AMS3}. \\
(b) From Proposition \ref{prop:2.6} is clear that the cuspidal quasi-support map 
$q\Psi_{Z_G (\sigma_0)}$ commutes with the operation of taking the contragredient of 
involved local system/representation. This is the same cuspidal quasi-support map as in
\cite[\S 5]{AMS1}, which forms the core of the construction of the cuspidal support map for
enhanced L-parameters in \cite[\S 7]{AMS1}. Hence that map commutes with taking contragredients
of enhancements of L-parameters, that is, with operation $(\phi',\rho') \mapsto (\phi', \rho'^\vee)$.

Consequently the set of enhanced L-parameters, for inner twists of $\mc G (F)$, whose cuspidal 
support is $\mc G^\vee$-conjugate to an element of $(\mc M^\vee, X_\nr^+ (\mc M (F)) \phi,\rho)$ 
is canonically in bijection with the analogous set involving $\rho^\vee$ instead of $\rho$.
\end{proof}

Assume that we have an algebra isomorphism
\[
\mh H_\tau \cong \mh H_{\phi,\rho} / (\mb r - \log (q_F) / 2) \quad \text{or} \quad
\mh H_\tau^{op} \cong \mh H_{\phi,\rho} / (\mb r - \log (q_F) / 2) .
\]
By \eqref{eq:5.1} there are isomorphisms
\[
\mh H_{\phi,\rho}^{op} \cong \mh H (G,M,q\cE)^{op} \cong \mh H (G,M,q\cE^\vee) \cong 
\mh H_{\phi,\rho^\vee} ,
\]
so \eqref{eq:8.13} boils down to
\begin{equation}\label{eq:8.13}
\mh H_\tau^{op} \cong \mh H_{\phi,\rho / \rho^\vee} / (\mb r - \log (q_F) / 2) ,
\end{equation}
where $\rho / \rho^\vee$ means that we have to pick one of the two.
Last but not least, we assume that the induced bijections
\begin{equation}\label{eq:8.14}
\Irr (\mc G (F))^{\tau +} \longleftrightarrow \Irr_{\mf a}(\mh H_\tau^{op}) \longleftrightarrow
\Irr_{\mf a} \big( \mh H_{\phi,\rho / \rho^\vee} / (\mb r - \log (q_F) / 2) \big)
\end{equation}
and Theorem \ref{thm:8.6} realize a local Langlands correspondence for $\Irr (\mc G (F))^{\tau +}$. 
This involves the para\-metri\-zation of irreducible and analytic standard modules from Theorem 
\ref{thm:4.3} and the translation to Langlands parameters in \cite[Theorem 3.8]{AMS3}.

In this setting, for $r = \log (q_F)/ 2$:
\begin{equation}\label{eq:8.17}
\mf g_N^{\sigma,-r} = \{ y \in \mf g_N : [\sigma,y] = -\log (q_F) y \} =
\{ y \in \mf g_N : \Ad (\exp \sigma) y = q_F^{-1} y \} .
\end{equation}
Here $(\exp \sigma,y)$ defines an unramified L-parameter $\phi : \mb W_F \rtimes \C \to G$, with
$\exp (\sigma)$ the image of a geometric Frobenius element $\Fr \in \mb W_F$. Via the construction of 
$G$ mentioned before Theorem \ref{thm:8.6} that gives rise to a Langlands parameter for $\mc G (F)$,
namely $\phi$ with the image of $\Fr$ adjusted by $\exp (\sigma)$.

We note that the above is based on the construction of $\mc G$ as inner twist of a quasi-split
$F$-group. Alternatively, one may work with $\mc G$ as rigid inner twist of a quasi-split
group \cite{Kal}. That requires some minor adjustments of the setup, which are discussed
in \cite[\S 7]{SolRamif}. In particular the above group $G$ will then become the centralizer of
$\phi (\mb W_F)$ in the complex dual group of $\mc G / Z(\mc G_\der)$.

\begin{thm}\label{thm:8.7}
We fix $r = \log (q_F)/2$.
\enuma{
\item Under the above assumptions, in particular \eqref{eq:8.13}--\eqref{eq:8.14}, Theorem 
\ref{thm:8.5} holds for $\Rep_{\mr{fl}}(\mc G (F))^{\tau +}$, where now $\mf g_N^{\sigma,-r}$ is 
a variety of Langlands parameters associated to $\Irr (\mc G (F))^{\tau \otimes \exp (\sigma)}$.
Thus the Kazhdan--Lusztig conjecture from \cite[Conjecture 8.11]{Vog} holds for 
irreducible and standard representations in $\Rep_{\mr{fl}} (\mc G (F))^{\tau +}$.
\item Part (a) holds unconditionally in the following cases:
\begin{itemize}
\item inner forms of general linear groups,
\item inner forms of special linear groups,
\item principal series representations of quasi-split groups,
\item unipotent representations (of arbitrary reductive groups over $F$),
\item classical $F$-groups -- namely symplectic groups, (special) orthogonal groups,
unitary groups and general (s)pin groups. Such a group need not be $F$-split, we only
require that it is a pure inner form of a quasi-split group.
\end{itemize}
}
\end{thm}
\begin{proof}
(a) When we assume \eqref{eq:8.13} with $\rho$, this is just a restatement of the above, taking 
into account that we omit the signs from \cite[Conjecture 8.11]{Vog}. When instead we assume
\eqref{eq:8.13} with $\rho^\vee$, we only have to add that the right hand side of Theorem
\ref{thm:8.5}.c does not change if replace both $\rho$ and $\rho'$ by their contragredients.\\
(b) We need to check that the setup involving \eqref{eq:8.13} and \eqref{eq:8.14} is valid in 
the mentioned cases. It suffices to check that for an analogous setup with affine Hecke algebras, 
because that can always be reduced to graded Hecke algebras with \cite[\S 2]{AMS3}.

For the inner forms of general/special linear groups, that was done in \cite{ABPSSLn} and \cite[\S 
5]{AMS3}. For unipotent representations we refer to \cite{Lus-Uni,Lus-Uni2,SolLLCunip,SolRamif}. 
For classical $F$-groups we use the LLC from \cite{MoRe}, the Hecke algebras for Bernstein
components from \cite{Hei} and the Hecke algebras for Langlands parameters as well as the
comparison results from \cite{AMS4}.

The required properties of affine Hecke algebras for principal series representations of split 
groups were established in \cite{ABPSprin,Roc}. This was generalized to quasi-split groups in
\cite{SolQS}.
\end{proof}

\appendix
\renewcommand{\theequation}{\Alph{section}.\arabic{equation}}

\section{Localization in equivariant cohomology}
\label{sec:B}

The fundamental localization theorem in equivariant (co)homology is \cite[Proposition 4.4]{Lus-Cusp2}.
It is analogous to theorems in equivariant K-theory \cite[\S 4]{Seg}, \cite[\S 5.10]{ChGi} and in
equivariant K-homology \cite[1.3.k]{KaLu}. In \cite{Lus-Cusp2} it is proven for equivariant local
systems $\mc L$ on varieties $X$ such that 
\begin{equation}\label{eq:B.1}
H_c^{\mr{odd}} (X,\mc L^\vee) = 0 .
\end{equation}
During our investigations it transpired that the condition \eqref{eq:B.1} is not always satisfied by
$(\ddot{\mf g}, \ddot{\mc L})$ in the setting of \cite[\S 8.12]{Lus-Cusp2} - on which important parts
of \cite{Lus-Cusp2} and other papers rely.
Fortunately, this can be repaired by relaxing the conditions of \cite[Proposition 4.4]{Lus-Cusp2},
as professor Lusztig kindly explained us. 

\begin{prop}\label{prop:B.1}
Let $G$ be a connected reductive complex group acting on an affine variety $X$. Let $M$ be a Levi
subgroup of $G$ (i.e. the centralizer of a semisimple element of $\mf g$). Then the natural map
\[
H^*_M (\pt) \underset{H^*_G (\pt)}{\otimes} H_*^G (X,\mc L) \longrightarrow H_*^M (X,\mc L)
\]
is an isomorphism.
\end{prop}
\begin{proof}
The proof of the analogous statement in equivariant K-homology \cite[1.8.(a)]{KaLu} can be 
translated to equivariant cohomology. The crucial point is the K\"unneth formula in equivariant
K-homology \cite[1.3.(n3)]{KaLu}, which is proven in \cite[\S 1.5--1.6]{KaLu}.

The conditions about simple connectedness in \cite[\S 1]{KaLu} are only needed to ensure that the
representation ring $R(T)$ is a free module over $R(G)$, for any maximal torus $T$ of $G$. 
In equivariant cohomology this translates to $H_T^* (\pt) \cong \mc O (\mf t)$ being free over 
\[
H^*_G (\pt) \cong \mc O (\mf g /\!/ G ) \cong \mc O (\mf t)^{W(G,T)}, 
\]
which is true for every connected reductive group $G$. 
\end{proof}

For $s \in \mf g$ we consider the inclusion $j : X^{\exp (\C s)} \to X$.

\begin{prop}\label{prop:B.2}
In the setup of Proposition \ref{prop:B.1}, assume that $s \in \mf g$ is central (and hence semisimple). 
The map
\[
\mr{id} \otimes j_! : \hat{H}^*_G (\pt)_{s} \underset{H^*_G (\pt)}{\otimes} H_*^G (X^{\exp (\C s)} 
,j^* \mc L) \longrightarrow \hat{H}^*_G (\pt)_{s} \underset{H^*_G (\pt)}{\otimes} H_*^G (X , \mc L) 
\]
is an isomorphism.
\end{prop}
\begin{proof}
Let $T \subset G$ be a maximal torus, so $s \in \mr{Lie}(T)$. The centrality of $s$ and 
Chevalley's theorem \cite[\S 4.9]{Var} entail that $\hat{H}^*_T (\pt)_{s}$ is a free module over 
$\hat{H}^*_G (\pt)_{s}$. By Proposition \ref{prop:B.1} 
\begin{align*}
\hat{H}^*_T (\pt)_s \underset{\hat{H}^*_G (\pt)_s}{\otimes} \hat{H}^*_G (\pt)_{s} 
\underset{H^*_G (\pt)}{\otimes} H_*^G (X ,\mc L) \; & \cong \;
\hat{H}^*_T (\pt)_s \underset{H^*_G (\pt)}{\otimes} H_*^G (X ,\mc L) \; \cong \\ 
\hat{H}^*_T (\pt)_s \underset{H^*_T (\pt)}{\otimes} H^*_T (\pt) 
\underset{H^*_G (\pt)}{\otimes} H_*^G (X ,\mc L) \; & \cong \;
\hat{H}^*_T (\pt)_s \underset{H^*_T (\pt)}{\otimes} H_*^T (X ,\mc L) ,
\end{align*}
and similarly with $(X^{\exp (\C s)}, j^* \mc L)$. In this way we reduce the issue from $G$ to $T$. 
That case is shown in \cite[Proposition 4.4.a]{Lus-Cusp2}.
\end{proof}

With Propositions \ref{prop:B.1} and \ref{prop:B.2} at hand, everything in \cite[\S 4]{Lus-Cusp2}
can be carried out without assuming that certain odd cohomology groups vanish. Instead we have
to assume that the involved groups are reductive, but that assumption can be lifted with 
\cite[\S 1.h]{Lus-Cusp1}. 

Problems with the condition \eqref{eq:B.1} also entail that \cite[5.1.(b,c,d)]{Lus-Cusp2} are
not necessarily isomorphisms in the setting of \cite[\S 8.12]{Lus-Cusp2}. Professor Lusztig showed
us that this can be overcome by rephrasing \cite[Proposition 5.2]{Lus-Cusp2} in the category
$\mc D^b_H (\tilde X)$ instead of $\mc D^b (\tilde X$), see \cite[\# 121]{Lus-Corr}.
The upshot is that all the proofs about representations of graded Hecke algebras in 
\cite[\S 8]{Lus-Cusp2} can be fixed.

\section{Compatibility with parabolic induction}
\label{sec:parabolic}

The family of geometric standard modules $E_{y,\sigma,r\rho}$ from Section \ref{sec:standard} behaves 
well under parabolic induction.
However, it does not behave as well as claimed in \cite[Theorem 3.4]{AMS2}: that result is slightly
too optimistic. Here we repair \cite[Theorem 3.4]{AMS2} by adding an extra condition, and we extend
it from graded Hecke algebras associated to a cuspidal support to graded Hecke algebras associated
to a cuspidal quasi-support. 

Let $P^\circ$ be a parabolic subgroup of $G^\circ$ with a Levi factor $L$. Let $v \in \mr{Lie}(L)$ 
be nilpotent and let $\cE$ be an $L$-equivariant cuspidal local system on $\mc C_v^L$. 
In \cite[\S 2]{AMS2}, a twisted graded Hecke algebra $\mh H (G,L,\cE)$ is associated to the cuspidal
support $(L,\mc C_v^L,\cE)$. Like in Condition \ref{cond:2.1}, we assume without loss of 
generality that $G = G^\circ N_G (P^\circ,\cE)$.

Let $Q$ be an algebraic subgroup of $G$ such that $Q^\circ = Q \cap G^\circ$ is a Levi subgroup
of $G$ and $L \subset Q^\circ$. Then $P^\circ Q^\circ$ is a parabolic subgroup of
$G^\circ$ with $Q^\circ$ as Levi factor. The unipotent radical $\mc R_u (P^\circ Q^\circ)$
is normalized by $Q^\circ$, so its Lie algebra $\mf u_Q = \mr{Lie} (\mc R_u (P^\circ Q^\circ))$
is stable under the adjoint actions of $Q^\circ$ and $\mf q$. In particular any $y \in \mf q$
acts on $\mf u_Q$ by the Lie bracket. We denote the cokernel of ad$(y) : \mf u_Q \to \mf u_Q$ by
$_y \mf u_Q$. For $N \in \mf u_Q$ and $(\sigma,r) \in Z_{\mf g \oplus \C}(y)$ we have
\[
[\sigma, [y,N]] = [y, [\sigma,N]] + [[\sigma,y], N] = 
[y, [\sigma,N]] + [2r y,N] \in \mr{ad}(y) \mf u_Q .
\]
Hence $\mr{ad}(\sigma)$ descends to a linear map ${}_y \mf u_Q \to {}_y \mf u_Q$. 
Following Lusztig \cite[\S 1.16]{Lus-Cusp3}, we define 
\[
\begin{array}{cccc}
\epsilon : & Z_{\mf q \oplus \C}(y) & \to & \C \\
& (\sigma,r) & \mapsto & \det( \mr{ad}(\sigma) - 2r : {}_y \mf u_Q \to {}_y \mf u_Q) 
\end{array}.
\]
It is easily seen that $\epsilon$ is invariant under the adjoint action of $Z_{Q \times \C^\times}(y)$, 
so it defines an element of $H^*_{Z_{Q \times \C^\times}(y)}(\{y\})$. For a given nilpotent $y$, 
all the parameters $(y,\sigma,r)$ for which parabolic induction from $\mh H (Q,L,\cE)$ to 
$\mh H (G,L,\cE)$ can behave problematically, are zeros of $\epsilon$. 

For any closed subgroup $S$ of $Z_{Q \times \C^\times}(y)^\circ$, restricting $\epsilon$ yields 
an element $\epsilon_S$ of $H^*_S (\{y\})$. We recall from \cite[Proposition 7.5]{Lus-Cusp1} 
that for connected $S$ there is a natural isomorphism 
\begin{equation}\label{eq:A.1}
H^S_* (\mc P_y, \dot{\cE}) \cong H_S^* (\{y\}) \underset{H^*_{Z^\circ_{G \times \C^\times}(y)}(\{y\})
}{\otimes} H^{Z^\circ_{G \times \C^\times}(y)}_* (\mc P_y, \dot{\cE}) .
\end{equation}
Here $H_S^* (\{y\})$ acts on the first tensor leg and $\mh H (G,L,\cE)$ acts on the second 
tensor leg. By \cite[Theorem 3.2.b]{AMS2} these actions commute, and $H_*^S (\mc P_y, \dot{\cE})$
becomes a module over $H_S^* (\{y\}) \otimes_\C \mh H (G,L,\cE)$.

To indicate that an object is constructed with respect to the group $Q$ (instead of $G$), we endow
it with a superscript $Q$. For instance, we have the variety $\mc P_y^Q$, which admits a natural map
\begin{equation}\label{eq:A.7}
\mc P_y^Q  \to \mc P_y : g (P^\circ \cap Q) \mapsto g P^\circ .
\end{equation}
Now we can formulate an improved version of \cite[Theorem 3.4]{AMS2}.

\begin{thm}\label{thm:A.1}
Let $S$ be a maximal torus of $Z^\circ_{Q \times \C^\times}(y)$.
\enuma{
\item The map \eqref{eq:A.7} induces an injection of $\mh H (G,L,\cE)$-modules
\[
\mh H (G,L,\cE) \underset{\mh H (Q,L,\cE)}{\otimes} H_*^S (\mc P_y^Q,\dot{\cE})
\to H_*^S (\mc P_y,\dot{\cE}) .
\]
It respects the actions of $H_S^* (\{y\})$ and its image contains 
$\epsilon_S H_*^S (\mc P_y,\dot{\cE})$.
\item Let $(\sigma,r) \in Z_{\mf q \oplus \C}(y)$ be semisimple, such that 
$\epsilon (\sigma,r) \neq 0$. The map \eqref{eq:A.7} induces an isomorphism of 
$\mh H (G,L,\cE)$-modules
\[
\mh H (G,L,\cE) \underset{\mh H (Q,L,\cE)}{\otimes} E^Q_{y,\sigma,r} 
\to E_{y,\sigma,r} ,
\]
which respects the actions of $\pi_0 (M^Q (y))_\sigma \cong \pi_0 (Z_Q (\sigma,y))$.
}
\end{thm}
\begin{proof}
(a) The given proof of \cite[Theorem 3.4]{AMS2} is valid with only one modification.
Namely, the diagram \cite[(25)]{AMS2} does not commute. A careful consideration of
\cite[\S 2]{Lus-Cusp3} shows that the failure to do so stems from the difference between certain
maps $i_!$ and $(p^*)^ {-1}$, where $p$ is the projection of a vector bundle on its
base space and $i$ is the zero section of the same vector bundle. In \cite[Lemma 2.18]{Lus-Cusp3}
this difference is identified as multiplication by $\epsilon_S$. \\
(b) Upon replacing $(\sigma, r)$ by a $Q^\circ$-conjugate element, we may assume that it lies in 
$\mr{Lie}(S)$. Then the proof of \cite[Theorem 3.4.b]{AMS2} needs only one small adjustment. 
From \eqref{eq:A.1} we get
\begin{align*}
\C_{\sigma,r} \underset{H_S^* (\{y\})}{\otimes} \epsilon_S H_*^S (\mc P_y, \dot{\cE}) & \; \cong \; 
\C_{\sigma,r} \underset{H_S^* (\{y\})}{\otimes} \epsilon_S H^*_S (\{y\}) \underset{
H_{M(y)^\circ}^* (\{y\})}{\otimes} H_*^{M (y)^\circ} (\mc P_y, \dot{\cE}) \\
& \; \cong \; \C_{\sigma,r} \underset{H_{M(y)^\circ}^* (\{y\})}{\otimes} H_*^{M (y)^\circ} 
(\mc P_y, \dot{\cE}) \; = \; E_{y,\sigma,r} .
\end{align*}
The difference with before is the appearance of $\epsilon_S$, with that and the above
the proof of \cite[Theorem 3.4.b]{AMS2} goes through.
\end{proof}

There is just one result in \cite{AMS2} that uses \cite[Theorem 3.4]{AMS2}, namely 
\cite[Proposition 3.22]{AMS2}. It has to be replaced by a version that involves only the cases 
of \cite[Theorem 3.4]{AMS2} covered by Theorem \ref{thm:A.1}.b. 

Now we set out to formulate and prove analogues of \cite[Theorem 3.4 and Proposition 3.22]{AMS2}
in the setting of Section \ref{sec:setup}, so with a cuspidal quasi-support $(M,\cC_v^M,q\cE)$.
Also, Condition \ref{cond:2.1} remains in force.

For comparison with Theorem \ref{thm:A.1} we assume that $M^\circ = L$ and that $\cE$ is contained
in the restriction of $q\cE$ to $\cC_v^L$. It is known from \cite[(93)]{AMS2} that
\begin{equation}\label{eq:A.10}
\mh H (G^\circ M, M, q\cE) = \mh H (G^\circ, L, \cE) .
\end{equation}
Taking this into account, \cite[(91)]{AMS2} provides a canonical isomorphism of \\
$\mh H (G,M,q\cE)$-modules
\begin{equation}\label{eq:A.11}
\begin{aligned}
H_*^{Z^\circ_{G \times \C^\times}(y)} (\mc P_y ,\dot{q\cE}) & \cong
\mr{ind}_{\mh H (G^\circ M, M, q\cE)}^{\mh H (G,M,q\cE)} H_*^{Z^\circ_{G^\circ M \times \C^\times}(y)} 
(\mc P_y^{G^\circ M} ,\dot{q\cE}) \\
& \cong \mr{ind}_{\mh H (G^\circ , L, \cE)}^{\mh H (G,M,q\cE)} 
H_*^{Z^\circ_{G^\circ \times \C^\times}(y)} (\mc P_y^{G^\circ} ,\dot{\cE}) .
\end{aligned}
\end{equation}
According to \cite[(95)]{AMS2} we can write $(q\cE )_v = \cE_v \rtimes \rho_M$ for a unique
\begin{equation}\label{eq:A.13}
\rho_M \in \Irr (\C [M_\cE / M^\circ, \natural_\cE]) = \Irr (\C [W_\cE / W_\cE^\circ, \natural_\cE]) .
\end{equation}
Next \cite[Lemma 4.4]{AMS2} says that, when $\sigma_0 \in \mf t$, the sets
\begin{equation}\label{eq:A.12}
\begin{aligned}
& \big\{ \rho^\circ \in \Irr \big( \pi_0 (Z (\sigma_0, y)) \big) : \Psi_{Z_G^\circ (\sigma_0)} 
(y,\rho^\circ) = [L,\cC_v^L,\cE ]_{G^\circ} \big\} , \\
& \big\{ \tau^\circ \in \Irr \big( \pi_0 (Z_{G^\circ M} (\sigma_0, y)) \big) : 
q\Psi_{M Z_G^\circ (\sigma_0)} (y,\tau^\circ) = [M,\cC_v^M,q\cE ]_{G^\circ M} \big\}
\end{aligned}
\end{equation}
are in bijection via $\rho^\circ \mapsto \rho^\circ \rtimes \rho_M$.
Further, by \cite[Lemma 4.5]{AMS2} the identification \eqref{eq:A.10} turns a standard 
$\mh H (G^\circ ,L,\cE)$-module $E^{G^\circ}_{y,\sigma,r,\rho^\circ}$ into the standard
$\mh H (G^\circ M,M,q\cE)$-module $E_{y,\sigma,r,\rho^\circ \rtimes \rho_M}$.

\begin{thm}\label{thm:A.4}
Let $Q$ be an algebraic subgroup of $G$ such that $Q^\circ = G^\circ \cap Q$ is a Levi subgroup of 
$G^\circ$ and $M \subset Q$. Let $y \in \mf q$ be nilpotent and let $S$ be a maximal torus of
$Z_{Q \times \C^\times}(y)$. Further, let $(\sigma,r) \in Z_{\mf q \oplus \C}(y)$ be semisimple,
such that $\epsilon (\sigma,r) \neq 0$.
\enuma{
\item The map \eqref{eq:A.7} induces an injection of $\mh H (G,M,q\cE)$-modules
\[
\mh H (G,M,q\cE) \underset{\mh H (Q,M,q\cE)}{\otimes} H_*^S (\mc P_y^Q,\dot{q\cE})
\to H_*^S (\mc P_y,\dot{q\cE}) .
\]
It respects the actions of $H_S^* (\{y\})$ and its image contains 
$\epsilon_S H_*^S (\mc P_y,\dot{q\cE})$.
\item The map \eqref{eq:A.7} induces an isomorphism of $\mh H (G,M,q\cE)$-modules
\[
\mh H (G,M,q\cE) \underset{\mh H (Q,M,q\cE)}{\otimes} E^Q_{y,\sigma,r} \to E_{y,\sigma,r} ,
\]
which respects the actions of $\pi_0 (Z_Q (\sigma,y))$.
}
Let $\rho \in \Irr \big( \pi_0 (Z_G (\sigma,y)) \big)$ with $q\Psi_{Z_G (\sigma_0)} (y,\rho) =
[M,\cC_v^M,q\cE ]_G$ and let\\ $\rho^Q \in \Irr \big( \pi_0 (Z_Q (\sigma,y)) \big)$ 
with $q\Psi_{Z_Q (\sigma_0)} (y,\rho) = [M,\cC_v^M,q\cE ]_Q$.
\enuma{\setcounter{enumi}{2}
\item There is a natural isomorphism of $\mh H (G,M,q\cE)$-modules
\[
\mh H (G,M,q\cE) \underset{\mh H (Q,M,q\cE)}{\otimes} E^Q_{y,\sigma,r,\rho^Q} \cong 
\bigoplus\nolimits_\rho \Hom_{\pi_0 (Z_Q (\sigma,y))} (\rho^Q ,\rho) \otimes 
E_{y,\sigma,r,\rho} ,
\]
where the sum runs over all $\rho$ as above.
\item For $r=0$, part (c) induces an isomorphism of 
$\mc O (\mf t \oplus \C) \rtimes \C[W_q\cE, \natural_{q\cE}]$-modules
\[
\mh H (G,M,q\cE) \underset{\mh H (Q,M,q\cE)}{\otimes} M^Q_{y,\sigma,0,\rho^Q} \cong 
\bigoplus\nolimits_\rho \Hom_{\pi_0 (Z_Q (\sigma,y))} (\rho^Q ,\rho) \otimes 
M_{y,\sigma,0,\rho} .
\]
\item The multiplicity of $M_{y,\sigma,r,\rho}$ in $\mh H (G,M,q\cE) 
\underset{\mh H (Q,M,q\cE)}{\otimes} E^Q_{y,\sigma,r,\rho^Q}$ is
$[\rho^Q : \rho]_{\pi_0 (Z_Q (\sigma,y))}$. \\
It already appears that many times as a quotient, via $E^Q_{y,\sigma,r,\rho^Q} \to 
M^Q_{y,\sigma,r,\rho^Q}$. More precisely, there is a natural isomorphism
\[
\Hom_{\mh H (Q,L,\cL)} (M^Q_{y,\sigma,r,\rho^Q}, M_{y,\sigma,r,\rho}) \cong
\Hom_{\pi_0 (Z_Q (\sigma,y))} (\rho ,\rho^Q ) .
\]
}
\end{thm}
\begin{proof}
(a) By \eqref{eq:A.1} and \cite[Lemma 3.3 and \S 4]{AMS2} the right hand side of the statement
is canonically isomorphic with
\begin{multline*}
H_S^* (\{y\}) \underset{H^*_{Z^\circ_{G \times \C^\times}(y)}}{\otimes} 
H_*^{Z^\circ_{G \times \C^\times}(y)} (\mc P_y, \dot{q\cE}) \; \cong \\
H_S^* (\{y\}) \underset{H^*_{Z^\circ_{G \times \C^\times}(y)}}{\otimes} \mh H (G,M,q\cE) 
\underset{\mh H (G^\circ M,M,q\cE)}{\otimes} H_*^{Z^\circ_{G^\circ M \times \C^\times}(y)} 
(\mc P_y^{G^\circ M}, \dot{q\cE}) . 
\end{multline*}
Via \eqref{eq:A.11} that is canonically isomorphic with
\begin{multline*}
H_S^* (\{y\}) \underset{H^*_{Z^\circ_{G \times \C^\times}(y)}}{\otimes} \mh H (G,M,q\cE) 
\underset{\mh H (G^\circ,L,\cE)}{\otimes} H_*^{Z^\circ_{G^\circ \times \C^\times}(y)}
(\mc P_y^{G^\circ}, \dot{\cE}) \\
\cong \; \mh H (G,M,q\cE) 
\underset{\mh H (G^\circ,L,\cE)}{\otimes} H_*^S (\mc P_y^{G^\circ}, \dot{\cE}) .
\end{multline*}
For similar reasons the left hand side of the statement is canonically isomorphic with
\begin{align*}
& \mh H (G,M,q\cE) \underset{\mh H (Q,M,q\cE)}{\otimes} \mh H (Q,M,q\cE) 
\underset{\mh H (Q^\circ M,M,q\cE)}{\otimes} H_*^S (\mc P_y^{Q^\circ M}, \dot{q\cE}) \; \cong \\
& \mh H (G,M,q\cE) \underset{\mh H (G^\circ M,M,q\cE)}{\otimes} \mh H (G^\circ M,M,q\cE) 
\underset{\mh H (Q^\circ M,M,q\cE)}{\otimes} H_*^S (\mc P_y^{Q^\circ M}, \dot{q\cE}) \; \cong \\
& \mh H (G,M,q\cE) \underset{\mh H (G^\circ,L,\cE)}{\otimes} \mh H (G^\circ,L,\cE) 
\underset{\mh H (Q^\circ,L,\cE)}{\otimes} H_*^S (\mc P_y^{Q^\circ}, \dot{\cE}) .
\end{align*}
Now we apply Theorem \ref{thm:A.1}.a for $G^\circ, Q^\circ, L,\cE$ and use the exactness of
$\mr{ind}^{\mh H (G,M,q\cE)}_{\mh H (G^\circ,L,\cE)}$.\\
(b) Like in part (a) there are canonical isomorphisms
\begin{align*}
E_{y,\sigma,r} \; & \cong \; \mh H (G,M,q\cE) \underset{\mh H (G^\circ M,M,q\cE)}{\otimes}  
E_{y,\sigma,r}^{G^\circ M} \; \cong \; 
\mh H (G,M,q\cE) \underset{\mh H (G^\circ,L,\cE)}{\otimes} E_{y,\sigma,r}^{G^\circ} ,\\
\mh H (G,& M,q\cE) \underset{\mh H (Q,M,q\cE)}{\otimes} E_{y,\sigma,r}^Q \\
& \cong \; \mh H (G,M,q\cE) \underset{\mh H (G^\circ M,M,q\cE)}{\otimes} \mh H (G^\circ M,M,q\cE)    
\underset{\mh H (Q^\circ M,M,q\cE)}{\otimes} E_{y,\sigma,r}^{Q^\circ M} \\
& \cong \; \mh H (G,M,q\cE) \underset{\mh H (G^\circ,L,\cE)}{\otimes} \mh H (G^\circ,L,\cE)    
\underset{\mh H (Q^\circ,L,\cE)}{\otimes} E_{y,\sigma,r}^{Q^\circ} .
\end{align*}
It remains to apply Theorem \ref{thm:A.1}.b.\\
(c,d,e) These can be shown in the same way as \cite[Proposition 3.22]{AMS2}, with the following
modifications:
\begin{itemize}
\item We use part (b) instead of \cite[Theorem 3.4.b]{AMS2}.
\item The references to \cite[\S 4]{AMS1} should be extended to the setting with cuspidal 
quasi-supports by means of \cite[\S 5]{AMS1}.
\item The references to \cite[\S 3]{AMS2} should be extended to the setting with cuspidal 
quasi-supports by invoking \cite[\S 4]{AMS2}. \qedhere
\end{itemize}
\end{proof}

Since $\epsilon$ is a nonzero polynomial function, its zero set is a subvariety (say of $V_y$) 
of smaller dimension. Still, we want to explicitly exhibit a large class of parameters 
$(y,\sigma,r)$ on which $\epsilon$ does not vanish. By \cite[Proposition 1.4.c]{AMS3} we may 
assume (via conjugation by an element of $G^\circ$) that $\sigma_0 ,\sigma - r \sigma_v \in \mf t$.

Let us call $x \in \mf t$ (strictly) positive with respect to $P Q^\circ$ if 
$\Re (\alpha (t))$ is (strictly) positive for all $\alpha \in R (\mc R_u (P Q^\circ),T)$.
We say that $x$ is (strictly) negative with respect to $P Q^\circ$ if $-x$ is (strictly) positive.

\begin{lem}\label{lem:A.2}
Let $y \in \mf q$ be nilpotent and let $(\sigma,r) \in \Sigma_v (\mf t \oplus \C)$ with
$[\sigma,y] = 2r y$. Write $\sigma = \sigma_0 + \textup d \gamma_y \matje{r}{0}{0}{-r}$
as in \eqref{eq:4.4}, with $\sigma_0 \in Z_{\mf t}(y)$. Assume that one 
of the following holds:
\begin{itemize}
\item $\Re (r) > 0$ and $\sigma_0$ is negative with respect to $P Q^\circ$;
\item $\Re (r) < 0$ and $\sigma_0$ is positive with respect to $P Q^\circ$;
\item $\Re (r) = 0$ and $\sigma_0$ is strictly positive or strictly negative
with respect to $P Q^\circ$.
\end{itemize}
Then $\epsilon (\sigma,r) \neq 0$.
\end{lem}
\begin{proof}
Via d$\gamma_y : \mf{sl}_2 (\C) \to \mf q$, $\mf u_Q$ becomes a finite dimensional
$\mf{sl}_2 (\C)$-module. Since $\sigma_0 \in \mf t$ commutes with $y$, it commutes with
d$\gamma_y (\mf{sl}_2 (\C))$. For any eigenvalue $\lambda \in \C$ of $\sigma_0$, let
$_\lambda \mf u_Q$ be the eigenspace in $\mf u_Q$. 

For $n \in \Z_{\geq 0}$ let $\mr{Sym}^n (\C^2)$ be the unique irreducible 
$\mf{sl}_2 (\C)$-module of dimension $n+1$. We decompose the $\mf{sl}_2 (\C)$-module 
$_\lambda \mf u_Q$ as
\[
_\lambda \mf u_Q = \bigoplus\nolimits_{n \geq 0} \mr{Sym}^n (\C^2)^{\mu (\lambda,n)} 
\quad \text{with} \quad \mu (\lambda,n) \in \Z_{\geq 0}.
\] 
The cokernel of $\matje{0}{1}{0}{0}$ on $\mr{Sym}^n (\C^2)$ is the lowest weight space
$W_{-n}$ in that representation, on which $\matje{r}{0}{0}{-r}$ acts as $-nr$. Hence
$\sigma$ acts on
\[
\mr{coker}(\mr{ad}(y) : \, _\lambda \mf u_Q \to \, _\lambda \mf u_Q) \cong
\bigoplus\nolimits_{n \geq 0} W_{-n}^{\mu (\lambda,r)} \quad \text{as} \quad 
\bigoplus\nolimits_{n \geq 0} (\lambda - nr) \mr{Id}_{W_n^{\mu (\lambda,r)}}.
\]
Consequently
\[
(\mr{ad}(\sigma) - 2r)\big|_{_\lambda \mf u_Q} = 
\bigoplus\nolimits_{\lambda \in \C} \bigoplus\nolimits_{n \geq 0} 
(\lambda - (n+2)r) \mr{Id}_{W_n^{\mu (\lambda,r)}}.
\]
By definition then 
\[
\epsilon (\sigma,r) = \prod\nolimits_{\lambda \in \C} 
\prod\nolimits_{n \geq 0} ( \lambda - (n+2) r )^{\mu (\lambda,n)} .
\]
When $\Re (r) > 0$ and $\sigma_0$ is negative with respect to $P Q^\circ$,
$\Re (\lambda - (n+2) r) < 0$ for all eigenvalues $\lambda$ of $\sigma_0$ on $\mf u_Q$,
and in particular $\epsilon (\sigma,r) \neq 0$.

Similarly, we see that $\epsilon (\sigma,r) \neq 0$ in the other two possible cases in the lemma.
\end{proof}

As an application of Lemma \ref{lem:A.2}, we prove a result in the spirit of the Langlands
classification for graded Hecke algebras \cite{Eve}. It highlights a procedure to obtain
irreducible $\mh H (G,M,q\cE)$-modules from irreducible tempered modules of a parabolic
subalgebra $\mh H (Q,M,q\cE)$: twist by a central character which is strictly positive
with respect to $P Q^\circ$, induce parabolically and then take the unique irreducible
quotient.

\begin{prop}\label{prop:A.3}
Let $y \in \mf g$ be nilpotent, $(\sigma,r) \in Z_{\mf g \oplus \C}(y)$ semisimple and let 
$\rho \in \Irr \big( \pi_0 (Z_G (\sigma,y)) \big)$ with $\Psi_{Z_G (\sigma_0)}(y,\rho) =
[M,\mc C_v^M,q\cE ]_G$.
\enuma{
\item If $\Re (r) \neq 0$ and $\sigma_0 \in i \mf t_\R + Z(\mf g)$, then
$M_{y,\sigma,r,\rho} = E_{y,\sigma,r,\rho}$.
\item Suppose that $\Re (r) > 0$ and $\sigma_0,\sigma - r \sigma_v \in \mf t$ such that 
$\Re (\sigma_0)$ is negative with respect to $P$. Then $\Re (\sigma_0)$ is strictly negative 
with respect to $P Q^\circ$, where $Q = Z_G (\Re(\sigma_0))$. Further $M_{y,\sigma,r,\rho}$ is 
the unique irreducible quotient of $\mh H (G,M,q\cE) \underset{\mh H (Q,M,q\cE)}{\otimes} 
M^Q_{y,\sigma,r,\rho}$.
\item In the setting of part (b), $\IM^* (M_{y,\sigma,r,\rho}) \cong 
\sgn^* (M_{y,-\sigma,-r,\rho})$ is the unique irreducible quotient of 
\[
\IM^* \big( \mh H (G,M,q\cE) \underset{\mh H (Q,M,q\cE)}{\otimes} M^Q_{y,\sigma,r,\rho} \big)
\cong \mh H (G,M,q\cE) \underset{\mh H (Q,M,q\cE)}{\otimes} \IM^* (M^Q_{y,\sigma,r,\rho}) .
\]
\item Let $(L,\cE)$ be related to $(M,q\cE)$ as in \eqref{eq:A.10}. Then 
$\IM^* (M^Q_{y,\sigma,r,\rho}) \cong \\ \sgn^* (M^Q_{y,-\sigma,-r,\rho})$ comes from the twist of 
a tempered $\mh H (Q^\circ,L,\cE)$-module by a strictly positive character of $\mc O (Z(\mf q))$.
}
\end{prop}
\textbf{Remark.} By \cite[(82)]{AMS2} the extra condition in part (a) holds for instance when
$\Re (r) > 0$ and $\IM^* (M_{y,\sigma,r,\rho})$ is tempered. By \cite[Proposition 1.7]{AMS2} 
every parameter $(y,\sigma)$ is $G^\circ$-conjugate to one with the properties as in (b).
\begin{proof}
(a) Write $\sigma_0 = \sigma_{0,\der} + z_0$ with $\sigma_{0,\der} \in \mf g_\der$ and 
$z_0 \in Z(\mf g)$. Then, as in the proof of \cite[Corollary 3.28]{AMS2},
\[
E^{G^\circ}_{y,\sigma,r,\rho^\circ} = E^{G^\circ}_{y,\sigma- z_0,r,\rho^\circ} \otimes \C_{z_0}
\quad \text{and} \quad 
M^{G^\circ}_{y,\sigma,r,\rho^\circ} = M^{G^\circ}_{y,\sigma- z_0,r,\rho^\circ} \otimes \C_{z_0} .
\]
By \cite[Theorem 1.21]{Lus-Cusp3} $E^{G^\circ}_{y,\sigma - z_0,r,\rho^\circ} = 
M^{G^\circ}_{y,\sigma - z_0,r,\rho^\circ}$ as $\mh H (G_\der,L \cap G_\der,\cE)$-modules, so
$E^{G^\circ}_{y,\sigma,r,\rho^\circ} = M^{G^\circ}_{y,\sigma,r,\rho^\circ}$ as 
$\mh H (G^\circ,L,\cE)$-modules. Together with \cite[Lemma 3.18 and (63)]{AMS2} 
this gives $E_{y,\sigma,r,\rho} = M_{y,\sigma,r,\rho}$.\\
(b) Notice that $Z_G (\sigma,y) = Z_Q (\sigma,y)$, so by \cite[Theorem 4.8.a]{AMS1} $\rho$ 
is a valid enhancement of the parameter $(\sigma,y)$ for $\mh H (Q,L,\cE)$.

By construction $\Re (\sigma_0)$ is strictly negative with respect to
$P Q^\circ$. Now Lemma \ref{lem:A.2} says that we may apply \cite[Proposition 3.22]{AMS2}. 
That and part (a) yield
\[
\mh H (G,L,\cE) \underset{\mh H (Q,L,\cE)}{\otimes} M^Q_{y,\sigma,r,\rho} =
\mh H (G,L,\cE) \underset{\mh H (Q,L,\cE)}{\otimes} E^Q_{y,\sigma,r,\rho} =
E_{y,\sigma,r,\rho} .
\]
Now apply \cite[Theorem 3.20.b]{AMS2}.\\
(c) This follows from part (b) and the compatability of $\IM^*$ with parabolic 
induction, as in \cite[(81)]{AMS2}. \\
(d) Write
\[
M^{Q^\circ}_{y,\sigma,r,\rho^\circ} = M^{Q^\circ}_{y,\sigma- z_0,r,\rho^\circ} \otimes \C_{z_0}
= M^{Q^\circ}_{y,\sigma- z_0,r,\rho^\circ} \otimes \big( \C_{z_0 - \Re (z_0)} \otimes
\C_{\Re (z_0)} \big)
\]
as in the proof of part (a), with $Q$ in the role of $G$. By \cite[Theorem 3.25.b]{AMS2}, \\
$M^{Q^\circ}_{y,\sigma- z_0,r,\rho^\circ} \otimes \C_{z_0 - \Re (z_0)}$ is anti-tempered.
The definition of $Q$ entails that $\Re (z_0)$ equals $\Re (\sigma_0)$, which we know is strictly 
negative. Hence 
\[
\IM^* \big( M^{Q^\circ}_{y,\sigma,r,\rho^\circ} \big) = \IM^* \big( M^{Q^\circ}_{y,\sigma- z_0,
r,\rho^\circ} \otimes \C_{z_0 - \Re (z_0)} \big) \otimes \C_{-\Re (\sigma_0)},
\]
where the right hand side is the twist of a tempered $\mh H (Q^\circ,L,\cE)$-module by the 
strictly positive character $-\Re (\sigma_0)$ of $S(Z(\mf q)^*)$. By \cite[(68) and (80)]{AMS2}
\begin{equation}\label{eq:A.6}
\IM^* (M^Q_{y,\sigma,r\rho}) = \IM^* \big( \tau \ltimes M^{Q^\circ}_{y,\sigma,r,\rho^\circ} \big) 
= \tau \ltimes \IM^* \big( M^{Q^\circ}_{y,\sigma,r,\rho^\circ} \big) . \qedhere
\end{equation}
\end{proof}
 
We note that by \cite[Lemma 3.16]{AMS2} $\mc O (Z(\mf q))$ acts on \eqref{eq:A.6} by the characters 
$\gamma (-\Re (\sigma_0))$ with $\gamma \in \Gamma_{q\cE}^Q$. Since $\Gamma_{q\cE}^Q$ normalizes 
$P Q^\circ$, it preserves the strict positivity of $-\Re (\sigma_0)$. In this sense 
$\IM^* (M^Q_{y,\sigma,r\rho})$ is essentially the twist of a tempered 
$\mh H (Q,L,\cE)$-module by a strictly positive central character.

\vspace{4mm}
\textbf{Declarations of interest:} none

\end{document}